\UseRawInputEncoding
\documentclass[12pt,psfig]{article}
\usepackage{float}
\usepackage{color}
\usepackage{subfigure}
\usepackage{amsmath}
\usepackage{amsthm}
\usepackage{amsfonts}
\usepackage{amssymb,latexsym}
\usepackage[all]{xy}
\usepackage{graphicx}
\usepackage[mathscr]{eucal}
\usepackage{verbatim}
\usepackage{hyperref}

\theoremstyle{plain}

    \newtheorem{thm}{Theorem}[section]
       
       \newtheorem{lem}{Lemma}[section]
       \newtheorem{rem}{Remark}[section]

\numberwithin{equation}{section}
\usepackage{graphicx}
\usepackage{pict2e}

\begin{document}
\title{Existence and dimensions of global attractors for a delayed reaction-diffusion equation on an unbounded domain}

\author{Wenjie Hu$^{1,2}$,  Tom\'{a}s
Caraballo$^{3,4}$\footnote{Corresponding author.  E-mail address: caraball@us.es (Tom\'as Caraballo).}, Alain Miranville$^{5,6}$
\\
\small  $^1$The MOE-LCSM, School of Mathematics and Statistics,  Hunan Normal University,\\
\small Changsha, Hunan 410081, China\\
\small $^2$Journal House, Hunan Normal University, Changsha, Hunan 410081, China\\
\small $^3$Dpto. Ecuaciones Diferenciales y An\'{a}lisis Num\'{e}rico, Facultad de Matem\'{a}ticas,\\
\small  Universidad de Sevilla, c/ Tarfia s/n, 41012-Sevilla, Spain\\
\small $^4$Department of Mathematics, Wenzhou University, \\
\small  Wenzhou, Zhejiang Province, 325035, China.\\
\small $^5$School of Mathematics and Information Science, Henan Normal University, \\
\small  Xinxiang, Henan Province, 453007, China.\\
\small $^6$Math\'{e}matiques, SP2MI, Universit\'{e} de Poitiers, \\
\small t\'{e}l\'{e}port 2, avenue Marie-et-Pierre-Curie, 86962 Futuroscope Chasseneuil cedex, France
}

\date {}
\maketitle

\begin{abstract}
The purpose of this paper is to investigate the existence and   Hausdorff dimension as well as fractal dimension  of global attractors for a delayed reaction-diffusion equation on an unbounded domain. The noncompactness of the domain causes the Laplace operator has a continuous spectrum, the semigroup generated by the linear part and the Sobolev embeddings are no longer compact, making the problem more difficult compared with the equations on bounded domains. We first obtain the existence of an absorbing set for the infinite dimensional dynamical system generated by the equation by a priori estimate of the solutions. Then, we show the asymptotic compactness of the solution semiflow by  an uniform a priori estimates for far-field values of solutions together with the Arzel\`a-Ascoli theorem, which facilitates us to show the existence of global attractors. By decomposing the solution into three parts and establishing a squeezing property of each  part,  we obtain the explicit upper estimation of both Hausdorff and fractal dimension of the global attractors, which  only depend on the inner characteristic of the equation, while not related  to the entropy number compared with the existing literature.
\end{abstract}
\bigskip


\section{Introduction}
Consider the following delayed reaction-diffusion equation on $\mathbb{R}^N$
\begin{equation}\label{1}
\left\{\begin{array}{l}\frac{\partial u}{\partial t}(x, t)= \Delta u(x, t)-\mu u(x, t)+\sigma u(x,t-\tau)+f \left(u(x,t-\tau)\right)+g(x), t>0, x \in \mathbb{R}^N,\\ u_0(x,s)=\phi(x, s),-\tau \leq s \leq 0, x \in \mathbb{R}^N,\end{array}\right.
\end{equation}
where, $N$ is a positive integer, $u(\cdot,t)\in \mathbb{X}\triangleq L^2(\mathbb{R}^N)$, the square Lesbegue integrable function on $\mathbb{R}^N$ with  usual inner product $(\cdot,\cdot)$ and norm $\|\cdot\|$, for any $t\geq0$. $\Delta$ is the Laplacian operator on $\mathbb{R}^N$, $\mu, \sigma$ and $\tau$ are positive constants, $g\in \mathbb{X}$. $u_t\in \mathcal{C}$ is defined by  $u_t(x,\theta)=u(x, t+\theta),-\tau \leq \theta \leq 0$, where $\mathcal{C}$ is the space of continuous functions from $[-\tau, 0]$ to $\mathbb{X}$ equipped with the supremum norm $\|\phi\|_{\mathcal{C}}=\sup_{\theta \in[-\tau, 0]}\|\phi(\theta)\|$ for any $\phi \in \mathcal{C}$.  The initial data $\phi\in \mathcal{C}$ and $f$ is a nonlinear operators on $\mathcal{C}$.  Equation \eqref{1} arises largely from the real world physical, chemical and biological processes with time delays. For instance, in the case $\sigma=0$ and $g\equiv 0$, \eqref{1} can be used to describe the evolution of the mature population of species, such as birds, whose immature individuals do not move around but the mature ones do and hence have drawn much attention from mathematical biology community in the past decades. Most existing works only concerned about the existence and  qualitative properties of traveling wave front solutions, which may explain the invasion of species. See, for instance, \cite{FT,GSA,MSL,SJZ} and the references therein.

Nevertheless, little attention has been paid to the asymptotic behavior of \eqref{1} due to the non-compactness of the spatial domain,  causing many nice results in dynamical system theory were ineffective. Recently, Yi, Chen and Wu \cite{YCWT} overcame this difficulty by innovatively  introducing the compact open topology and showed the existence and global attractivity of a positive steady state by delicately constructing a priori estimate for nontrivial solutions under the compact open topology in the case $\sigma=0, g\equiv 0$ and the nonlinear term $f$ is globally bounded and admits a unique two periodic fixed point. The methods have  been generalized to investigate the existence and structure of attractors for deterministic and stochastic nonlocal delayed reaction-diffusion equations on a semi-infinite domain with a Dirichlet boundary condition at the finite end in \cite{YZD,HC} and have been adopted to tackle similar problems with Neumann boundary condition in \cite{HDZ}.

One natural question arises. What can we say about the dynamics of \eqref{1} if we do not impose the conditions on $f$ used in \cite{YCWT} and work in the  natural phase space $\mathcal{C}$ under the usual supremum norm? Since \eqref{1} generates an infinite dimensional dynamical system in $\mathcal{C}$, as a first step, if we can obtain the existence of attractors for \eqref{1}, then its essential dynamics can be reduced to a compact set. Furthermore,  if the attractors have  finite dimensionality, then one is likely to apply the analysis and computation tools of finite dimension dynamical systems to study the complex structure of the attractors and hence give a global picture of dynamics of \eqref{1}. Indeed, the study of existence and estimation of  Hausdorff dimension as well as fractal dimension  of  attractors for both deterministic or random infinite dimensional dynamical systems generated by  evolution equations has drawn much attention from both dynamical system and applied physics community in the past decades due to the significant roles they  played  in the study of the long time behavior of nonlinear dynamical systems.

The theory of  existence of attractors for deterministic infinite dimensional dynamical systems has been well established, see the monographs \cite{BV,HJ,TR}.  Criteria for the  finite Hausdorff dimensionality of attractors for deterministic fluid dynamics models were firstly derived by  Douady and  Oesterle \cite{DO}, which was later generalized by  Constantin,  Foias and  Temam \cite{T6,CFT} (see also the book of  Temam \cite{TR}). Since then, the methods have been widely and extensively adopted to investigate dimensions of global attractors for various partial differential equations (PDEs) on different domains. For partial differential equations on bounded domains, \cite{BPV} studied the existence and dimension estimation of attractors for autonomous N-S equations, \cite{CFT88} studied the dimension of the attractors in two-dimensional turbulence.  Moreover, \cite{CV94} extended the results to the nonautonomous case. There is a huge literature in this field to be covered here and more works can be found in monographs \cite{BV,HJ,TR} and the references therein.

On the other hand, when the domain is unbounded, several problems arise. As pointed in \cite{CE00}, the Laplace operator has a continuous spectrum, $H^1(\mathbb{R}^n)$ is not compactly embedded in $L^2(\mathbb{R}^n)$ and the solution semiflow does not have compact absorbing sets   in the original topology, causing the method of obtaining compactness, by a priori estimates and compactness of the Sobolev embeddings for the bounded case,   ineffective. In  order to overcome these difficulties,  \cite{BV90} established a method to study the existence and dimensions estimation  of attractors for the autonomous reaction diffusion equations on unbounded domain in a weighted space, which was then adopted by \cite{EM99} to explore dimensions  of reaction diffusion equations, used by \cite{FLS96} to explore global attractors for degenerate parabolic equations and extended to nonautonomous case in \cite{CE00}. However, when studying in weighted spaces one has to impose an additional condition ensuring that the initial data and forcing term also belong to the corresponding spaces \cite{GR02}. Thus, \cite{W99} established a new method to obtain  the results in the usual Hilbert space based on  uniform a priori estimates for far-field values of solutions following the ideas in \cite{B97} and \cite{B04}, where global attractors of generalized semiflows and the Navier-Stokes equations as well as the damped semilinear wave equations were studied.

The above mentioned theory and methods are mainly established for PDEs in Hilbert spaces. Nevertheless, the natural phase space $\mathcal{C}$ of \eqref{1} is a Banach space. The lack of smooth inner product makes the above theoretical results cannot be directly applied. Thus, for the purpose of obtaining the existence and  dimensions estimation of global attractors for delayed partial differential equations,  So and Wu \cite{SW91} recast general semilinear partial functional differential equations into a product Hilbert space and obtained  topological dimensions of attractors by a semigroup approach. Su and Qin \cite{QY} also recast 2D Navier-Stokes-Voight equations with a distributed delay in a product Hilbert space to establish  the existence and estimation of Hausdorff dimensions of attractors by virtue of energy estimates as well as compact embedding technique. Both works consider equations on bounded domains and recast functional differential equations in an auxiliary product Hilbert space, which  is unnatural and may cause redundancy in dimensions estimation.

In our current work, we investigate the existence as well as their  Hausdorff  dimension and fractal dimension estimation of global attractors for \eqref{1}  directly in the natural phase space, i.e.,  the Banach space $\mathcal{C}$.  In order to show the existence of global attractors,  we show the asymptotic  compactness of the solutions semiflow by uniform a priori estimates for far-field values of solutions. For the purpose of estimating the dimensions of obtained attractors, we decompose the solution of \eqref{1} into a sum of three parts, among which, each part fulfills  squeezing property. The idea of decomposing the solutions originates from \cite{ZM}, where fractal dimension of attractors for stochastic non-autonomous reaction-diffusion equation in $\mathbb{R}^3$ was studied. Unlike \cite{ZM}, where orthogonal projectors with finite rank in Hilbert space together with variation techniques are adopted to analyze the finite dimension of the first part and the decay of the other two parts, we adopt the phase space decomposition based on the exponential dichotomy   to overcome barrier caused by the lack of smooth inner product in a Banach space.  Here, our approach  gives explicit bounds that only depend on the inner characteristics of equation \eqref{1}, while not related   to the entropy number as \cite{C9, C17,C18,C21,C20,C25} did.

The remaining of this paper is organized as follows. Section 2 is devoted to preliminaries, including notation, definitions and some lemmas. In  Section 3,  we firstly give some a priori estimate of the solution on a unbounded domain and a priori estimates for far-field values of solutions, which shows the existence of absorbing sets.  Then, we prove asymptotic compactness of the solution semiflow in Section 4, which combined with the existence of absorbing sets implies the existence of global attractors. In Section 5, we construct squeeze property by decomposing the solution into three parts, two on a bounded domain while the third on an infinite domain. Subsequently, we give explicit upper bound of the Hausdorff  and fractal dimensions of the obtained  attractors in Section 6 by the squeezing property established in Section 5.  At last, we summarize the paper by making a detailed comparison of the present work with the methods established by Chepyzhov, Efendiev, Miranville and Zelik  in their  ground breaking pioneer works \cite{CE00,EM99,EM01,C18} and point out some potential directions for future study.

\section{Preliminaries}
We first introduce more notation used throughout the remaining part of this paper. Define the Hilbert space $\mathbb{X}^1$ as  $\mathbb{X}^1\triangleq \{\varphi\in \mathbb{X}| \frac{\partial \varphi(x) }{x_i}\in \mathbb{X}, i=1,2, \cdots, N, x\in \mathbb{R}^N\}$. For convenience, we take the seminorm $\|\varphi\|_{\mathbb{X}^1}=[\int_{\mathbb{R}^N} \nabla\varphi(x)^2dx]^{\frac{1}{2}}, \varphi\in \mathbb{X}^1$  as the norm of $\mathbb{X}^1$ since it is in fact a norm equivalent to the usual norm of $\mathbb{X}^1$.  Denote by $\mathcal{C}^1=C([-\tau, 0],\mathbb{X}^1)$ the set of all  continuous functions from $[-\tau, 0]$ to $\mathbb{X}^1$ equipped with the usual supremum norm $\|\phi\|_{\mathcal{C}^1} =\sup\{\| \phi(\xi)\|_{\mathbb{X}^1} :\xi \in [-\tau,0]\}$ for all  $\phi\in \mathcal{C}^1$.

 For a given $0<K<+\infty$, denote the ball centered at 0 with radius $K$ by $\Omega_K= \{x \in \mathbb{R}^N:|x|
 <K \}$, its boundary $\partial \Omega_K$ by $\partial \Omega_K=\left\{x \in \mathbb{R}^N:|x|=K\right\}$ and its complement by $\Omega_K^C=\left\{x \in \mathbb{R}^N:|x|\geq K\right\}$ respectively.  Let $\mathbb{X}_{\Omega_K}$ be the Hilbert space $\mathbb{X}_{\Omega_K}=\{\phi \in L^2(\Omega_K)| \phi(\partial \Omega_K)=0\}$, where $L^2(\Omega_K)$ is the square Lesbegue integral functions on $\Omega_K$. Define $\mathbb{X}_{\Omega_K}^1$ as  $\mathbb{X}_{\Omega_K}^1\triangleq \{\varphi\in \mathbb{X}_{\Omega_K}| \frac{\partial \varphi(x) }{x_i}\in L^2(\Omega_K), i=1,2, \cdots, N, x\in \Omega_K\}$. For convenience, we take the seminorm $\|\varphi\|_{\mathbb{X}_{\Omega_K}^1}=[\int_{\Omega_K} \nabla\varphi(x)^2dx]^{\frac{1}{2}}, \varphi\in \mathbb{X}_{\Omega_K}^1$  as the norm of $\mathbb{X}_{\Omega_K}^1$ since it is in fact a norm equivalent to the usual norm of $\mathbb{X}_{\Omega_K}^1$.  Denote by $\mathcal{C}_{\Omega_K}=C([-\tau, 0],\mathbb{X}_{\Omega_K})$ and $\mathcal{C}_{\Omega_K}^1=C([-\tau, 0],\mathbb{X}_{\Omega_K}^1)$ the set of all  continuous functions from $[-\tau, 0]$ to $\mathbb{X}_{\Omega_K}$ and $\mathbb{X}_{\Omega_K}^1$ equipped with the usual supremum norm $\|\phi\|_{\mathcal{C}_{\Omega_K}} =\sup\{\| \phi(\xi)\|_{\mathbb{X}_{\Omega_K}} :\xi \in [-\tau,0]\}$ and $\|\phi\|_{\mathcal{C}_{\Omega_K}^1} =\sup\{\| \phi(\xi)\|_{\mathbb{X}_{\Omega_K}^1} :\xi \in [-\tau,0]\}$ for all  $\phi\in \mathcal{C}_{\Omega_K}$ and $\phi\in \mathcal{C}_{\Omega_K}^1$ respectively.

By the  Fourier transformation \cite{YCWT}, we can see that the semigroup  $S(t): \mathbb{X}\rightarrow \mathbb{X}$ for any $t\geq 0$ generated by $\Delta-\mu I$ is defined as
 \begin{equation}\label{2.1}
\left\{\begin{array}{l}
S(0)[\phi](x)=\phi(x), \\
S(t)[\phi](x)=\frac{e^{-\mu t}}{ (4 \pi t)^{N/2}} \int_{\mathbb{R}^N} \phi(y) e^{ -\frac{|x-y|^{2}}{4 t} } dy, t \in(0, \infty),
\end{array}\right.
 \end{equation}
for $(x, \phi) \in \mathbb{R}^N \times \mathbb{X}$, which is  analytic and strongly continuous on $\mathbb{X}$.
For later use, we introduce the following results concerning the properties of semigroup $\{S(t)\}_{t\geq 0}$. The details of the proof can be found in \cite[Lemmas 2.1, 2.4]{YCWT}.
\begin{lem}\label{lem2.1a}
Let $\{S(t)\}_{t\geq 0}$  be defined by \eqref{2.1}, then we have the following results.\\
(i) $\|S(t) \phi\|\leq e^{-\mu t} \|\phi\|$ for all $\phi \in \mathbb{X}$, $t \in \mathbb{R}_{+}$.\\
(ii) $\{S(t)\}_{t\geq 0}$ is an analytic and strongly continuous semigroup on $\mathbb{X}$.\\
(iii) For all $t \in(0, \infty)$ and $(x, \phi) \in (0, \infty)\times \mathbb{X},$ there holds
$$
\begin{array}{l}
 S(t)[a](x)=a e^{-\mu t}.
\end{array}
$$
\end{lem}
In the remaining part, we always assume that the nonlinear term $f$ satisfies $f(\mathbf{0})=0$ and the following Lipschitz conditions.

$\mathbf{Hypothesis\  A1}$
$\left\|f\left(\phi_1\right)-f\left(\phi_2\right)\right\| \leq L_f\left\|\phi_1-\phi_2\right\|_{\mathcal{C}} \text { for any } \phi_1, \phi_2 \in \mathcal{C}.$

 By an argument of steps, we know that for any given $\phi \in \mathcal{C}$, \eqref{1} has a unique solution in $\mathcal{C}$ for all $t\geq 0$. By methods of variation of constant, \eqref{1} is equivalent to the following integral equation with the given initial function
  \begin{equation}\label{2.2a}
    \left\{
     \begin{array}{ll}
     \displaystyle u(t)=S(t)\phi(0) +\sigma \int_{0}^{t}S(t-s) u(s-\tau) \mathrm{d}s+ \int_{0}^{t}S(t-s) [f(u(s-\tau))+g]\mathrm{d}s,   t>0, \\u_{0}=\phi\in \mathcal{C}.
     \end{array}
     \right.
  \end{equation}

Following similar techniques as  \cite[Theorem 8]{WK}, we have the following results on the existence of solutions.
\begin{lem}\label{lemma4.2} Assume that $\mathbf{Hypothesis\  A1}$ holds. Then, \\
(i) for any $\phi \in \mathcal{C}$, there exists a solution $u(\cdot)$ to problem \eqref{1} with $u \in L^2(-\tau, T ; \mathbb{X}) \cap L^2(0, T ; \mathbb{X}^1) \cap L^{\infty}(0, T ; \mathbb{X}) \cap C([-\tau, T] ; \mathbb{X}), \forall T>0$,\\
(ii) for any $\phi \in C([-\tau, T] ; \mathbb{X}^1)$, problem \eqref{1} admits a strong solution
$$
u \in L^2\left(0, T ; \mathbb{X}\right) \cap C([-\tau, T] ; \mathbb{X}^1), \quad \forall T>0.
$$
\end{lem}
Let the functions $\chi_{\Omega_K}$ and $\chi_{\Omega_K^C}$ be the characteristic functions on $\Omega_K$ and  $\Omega_K^C$ respectively, that is
\begin{equation}\label{2.1}
\chi_{\Omega_K}(x)=\left\{\begin{array}{l}0, x \in \Omega_K^C,\\ 1,x \in \Omega_K,\end{array}\right.
\end{equation}
and
\begin{equation}\label{2.2}
\chi_{\Omega_K^C}(x)=\left\{\begin{array}{l}0, x \in \Omega_K,\\ 1, x \in \Omega_K^C.\end{array}\right.
\end{equation}
 Set $u(x,t)=v(x,t)+w(x,t)$, with $v(x,t)=u(x,t) \chi_{\Omega_K}(x)$ and $w(x,t)=u(x,t) \chi_{\Omega_K^C}(x)$
for any $t \geq -\tau$ and $x \in \mathbb{R}^N$, then $v(t, x)$ and $w(t, x)$ satisfy
\begin{equation}\label{2.3}
\left\{\begin{array}{l}
\frac{\partial v(x,t)}{\partial t}=\Delta v(x,t)-\mu v(x,t)+\sigma v(x,t-\tau)+f \left(u(x,t-\tau)\right)\chi_{\Omega_K}(x)+g(x)\chi_{\Omega_K}(x), \\
v(x,s)=\phi(x, s) \chi_{\Omega_K}(x)\triangleq \varphi, s \in[-\tau, 0], x \in \mathbb{R}^N,\\
v(x,t)=0, t \in(0, \infty), x \in \partial \Omega_K
\end{array}\right.
\end{equation}
and
\begin{equation}\label{2.4}
\left\{\begin{array}{l}
\frac{\partial w(x,t)}{\partial t}=\Delta w(x,t)-\mu w(x,t)+\sigma w(x,t-\tau)+f \left(u(x,t-\tau)\right)\chi_{\Omega_K^C}(x)+g(x)\chi_{\Omega_K^C}(x), \\
w(x,s)=\phi(x, s) \chi_{\Omega_K^C}(x)\triangleq \psi, s \in[-\tau, 0], x \in \mathbb{R}^N,
\end{array}\right.
\end{equation}
respectively, in which $u$ is the solution to \eqref{1}.

For later use, we  introduce the following lemma from \cite{HJ} concerning the existence of global attractors for continuous semigroups and finite cover of finite dimensional subspace $F$ of a Banach space $X$.
\begin{lem}\label{lem2.1}
Let $X$ be a Polish space and $S(t): X \rightarrow X$  be a continuous semigroup for all $t \geqslant 0$ which is asymptotically compact. If $S(t)$ admits a bounded  absorbing set  $\mathcal{B}\subseteq X$, then is has a global attractor $ \mathcal{A}\subseteq \mathcal{B}$, which is given by
$$
\mathcal{A}=\bigcap_{s\geq 0} \overline{\bigcup_{t \geqslant s} S(t) \mathcal{B}}.
$$
\end{lem}

For a finite dimensional subspace $F$ of a Banach space $X$, denote by $B^F_r(x)$ the ball in $F$ of center $x$ and radius $r$, that is $B^F_r(x)=\{y \in F |\|y-x\|\leq r\}$. It is proved in \cite{30} that the following covering lemma of balls in finite dimensional Banach spaces is true.
\begin{lem}\label{lem2.2}
For every finite dimensional subspace $F$ of a Banach space $X$, we have
 \begin{equation}
N\left(r_1, B_{r_2}^F\right)\leq m 2^m\left(1+\frac{r_1}{r_2}\right)^m,
\end{equation}
for all $r_1>r_2>0$, where $m=\operatorname{dim} F$ and $N\left(r_1, B_{r_2}^F\right)$ is the minimum number of balls needed to cover  the ball of radius $r_1$ by balls of $B_{r_2}^F$  calculated in the Banach space $X$.
\end{lem}
\section{Uniform estimates of solutions}
Let $u^\phi \in C([-\tau, \infty), \mathbb{X})$ be the  global solution obtained in Lemma \ref{lem2.1}. Define the dynamical system $\Phi: \mathbb{R}_+\times \mathcal{C}\rightarrow \mathcal{C}$ generated by the solution of \eqref{1}   as $\Phi(t, \phi)=u_t^\phi$. The following lemma shows that the dynamical system $\Phi$ possesses an absorbing set.
\begin{lem}\label{lem2.2}
Assume that  $\sigma (L_f+1) e^{\mu \tau}-\mu<0$. Then,  the set $\mathcal{K} \triangleq \{ \psi\in \mathcal{C}|\|\psi\|_{\mathcal{C}} \leq2(\frac{\|g\|}{\mu}+\frac{\|g\|\beta}{\mu(\mu-\beta)})\}$ is an absorbing set of the dynamical system $\Phi$ generated by \eqref{1}. That is, for any bounded set $\mathcal{D}\subseteq \mathcal{C}$, there is a $T_{\mathcal{D}}>0$ such that
$$
 \Phi\left(t, \mathcal{D}\right)\subseteq \mathcal{K}
$$
for all $t>T_{\mathcal{D}}$.
\end{lem}
\begin{proof}
It follows from \eqref{2.2a}, Lemma \ref{lem2.1a} and boundedness of $f$ that
\begin{equation}\label{2.4}
\begin{aligned}
\left\|u(t)\right\|\leq &\left\| S(t)\phi(0) \right\|+\|\sigma \int_{0}^{t}S(t-s) u(s-\tau) \mathrm{d}s\|+\|\int_{0}^{t}S(t-s)[ f(u(s-\tau,\cdot))+g]\mathrm{d}s\| \\
\leq & e^{-\mu t}\left\|\phi(0)\right\|+ \sigma \int_{0}^{t} e^{-\mu(t-s) } \|u(s-\tau)\| \mathrm{d} s+\int_{0}^{t} e^{-\mu(t-s) } [L_f\|u(s-\tau,\cdot)\|+\|g\|] \mathrm{d} s\\
\leq & e^{-\mu t}\left\|\phi(0)\right\|+\sigma \int_{0}^{t} e^{-\mu(t-s) }(L_f+1)\|u(s-\tau)\| \mathrm{d} s+\frac{\|g\|(1-e^{-\mu t})}{\mu}.
\end{aligned}
\end{equation}
Thus, for $\xi\in [-\tau, 0]$, we have
\begin{equation}\label{2.5}
\begin{aligned}
\left\|u(t+\xi)\right\|
\leq & e^{-\mu (t+\xi)}\left\|\phi(0)\right\|+\sigma \int_{0}^{t+\xi}(L_f+1) e^{-\mu(t+\xi-s) } \|u(s-\tau)\| \mathrm{d} s+\frac{\|g\|(1-e^{-\mu (t-\xi)})}{\mu}\\
\leq & e^{\mu \tau}e^{-\mu t }\left\|\phi(0)\right\|+\sigma(L_f+1) e^{\mu \tau}\int_{0}^{t} e^{-\mu(t-s) } \|u(s-\tau)\| \mathrm{d} s+\frac{\|g\| }{\mu}.
\end{aligned}
\end{equation}
Keep in mind that $\left\|u_{t}\right\|_{\mathcal{C}}=\sup \left\{\|u(t+\xi)\|: \xi \in[-\tau, 0]\right\},$ by which we have
\begin{equation}\label{2.6}
\begin{aligned}
\left\|u_t\right\|_{\mathcal{C}}\leq & e^{\mu \tau}e^{-\mu t }\left\|\phi\right\|_{\mathcal{C}}+\sigma(L_f+1)  e^{\mu \tau}\int_{0}^{t} e^{-\mu(t-s) } \|u_s\|_{\mathcal{C}} \mathrm{d} s+\frac{\|g\|}{\mu}.
\end{aligned}
\end{equation}
Multiply both sides of \eqref{2.6} by $e^{\mu t}$ gives
\begin{equation}\label{2.7}
\begin{aligned}
 e^{\mu t}\left\|u_t \right\|\leq & e^{\mu \tau}\left\|\phi\right\|_{\mathcal{C}}+\sigma (L_f+1) e^{\mu \tau}\int_{0}^{t} e^{ \mu s } \|u_s\|_{\mathcal{C}} \mathrm{d} s+\frac{Me^{\mu t}}{\mu}.
\end{aligned}
\end{equation}
Applying Gr\"{o}nwall's inequality yields
\begin{equation}\label{2.8}
\begin{aligned}
e^{\mu t} \left\|u_t\right\|_{\mathcal{C}}\leq &(\frac{\|g\|}{\mu}e^{\mu t}+e^{\mu \tau}\left\|\phi\right\|_{\mathcal{C}})+\frac{\|g\|\beta e^{\beta t}}{\mu(\mu-\beta)}[e^{(\mu-\beta)t}-1]+
 (e^{\beta t}-1) e^{\mu \tau} \left\|\phi\right\|_{\mathcal{C}},
\end{aligned}
\end{equation}
where $\beta=\sigma (L_f+1) e^{\mu \tau}$. Hence,
\begin{equation}\label{2.9}
\begin{aligned}
  \left\|u_t\right\|_{\mathcal{C}}\leq & (\frac{\|g\|}{\mu}+e^{\mu (\tau-t)}\left\|\phi\right\|_{\mathcal{C}})+\frac{\|g\|\beta}{\mu(\mu-\beta)}[1-e^{-(\mu-\beta)t}]+
 e^{\mu \tau} \left\|\phi\right\|_{\mathcal{C}}e^{(\beta-\mu) t}\\
 \leq& (\frac{\|g\|}{\mu}+\frac{\|g\|\beta}{\mu(\mu-\beta)})+(e^{\mu (\tau-t)}\left\|\phi\right\|_{\mathcal{C}}+e^{\mu \tau} \left\|\phi\right\|_{\mathcal{C}}e^{(\beta-\mu) t}).
 \end{aligned}
\end{equation}
Since we have assumed that $\sigma (L_f+1) e^{\mu \tau}-\mu<0$, one can see, for any bounded set $ \mathcal{D} \subset \mathcal{C}$, there exists a $T_{\mathcal{D}}>0$ such that for all $t\geq T_{\mathcal{D}}$
\begin{equation}\label{2.9}
\begin{aligned}
 e^{\mu (\tau-t)}\left\|\phi\right\|_{\mathcal{C}}+e^{\mu \tau} \left\|\phi\right\|_{\mathcal{C}}e^{(\beta-\mu) t}\leq \frac{\|g\|}{\mu}+\frac{\|g\|\beta}{\mu(\mu-\beta)},
 \end{aligned}
\end{equation}
indicating that
\begin{equation}\label{2.10}
\begin{aligned}
\left\|u_t\right\|_{\mathcal{C}} \leq & 2(\frac{\|g\|}{\mu}+\frac{\|g\|\beta}{\mu(\mu-\beta)}).
\end{aligned}
\end{equation}
That is, $\mathcal{K}$ is an absorbing set for $\Phi$.
This completes the proof.
\end{proof}

Next, we establish some estimations of the integral of the solution $u^\phi_{t}(\cdot)$ to equation \eqref{1} in $\mathcal{C}^{1}$, which is important to obtain estimations on $u^\phi_{t}(\cdot)$ in $\mathcal{C}^{1}$.

\begin{lem}\label{lem2.3}  Assume that $\mathbf{Hypothesis \  A1}$ and assumptions of Lemma \ref{lem2.2} hold and let $T_{\mathcal{D}}$ be defined in Lemma \ref{lem2.2}. Then,  the solution $u^\phi_{t}(\cdot)$ to equation \eqref{1}  satisfies
\begin{equation}\label{3.15}
\begin{aligned}
\int_{t}^{t+1}\left\|u^\phi_{s}\left(\cdot\right)\right\|_{\mathcal{C}^{1}}^{2} \mathrm{d} s \leq  c_4,
\end{aligned}
\end{equation}
for all $t>T_{\mathcal{D}}$, where $c_4=\frac{1}{2}c_2\|\phi(0)\|^{2}+\frac{(\sigma+L_{f}^{2})c_3^2}{\mu-\sigma-1}+ \frac{c_1}{2}$, $c_1=\frac{1}{\mu-\sigma-1}(\frac{1}{\mu} \| g \|^{2}+2\| f(0) \|^{2})$, $c_2= e^{(\mu-\sigma-1)\tau}$ and $c_3=2(\frac{\|g\|}{\mu}+\frac{\|g\|\beta}{\mu(\mu-\beta)})$.
\end{lem}
\begin{proof}
Taking inner product of each term of \eqref{1} with $u(t)$ in $\mathbb{X}$, we obtain
\begin{equation}\label{3.1}
\begin{aligned}
& \frac{1}{2} \frac{\mathrm{d}\|u(t)\|^{2}}{\mathrm{d} t}=-\|\nabla u(t)\|^{2}-\mu \|u(t)\|^{2}+\int_{\mathbb{R}^N} \sigma u(t-\tau) u(t) \mathrm{d}x +\int_{\mathbb{R}^N} f(u(t-\tau)) u(t) \mathrm{d}x+\int_{\mathbb{R}^N} g u(t) \mathrm{d}x.
\end{aligned}
\end{equation}
We now estimate each term on the right-hand side of \eqref{3.1}.
By the basic inequality $2ab\leq a^2+b^2$ and  $\mathbf{Hypothesis \    A1}$,  we deduce
\begin{equation}\label{3.2a}
\begin{aligned}
\sigma \int_{\mathbb{R}^N}u(t-\tau)u(t) \mathrm{d} x & \leq \sigma[\frac{1}{2} \| u_t \|_{\mathcal{C}}^{2} +\frac{1}{2}\|u(t)\|^{2}],
\end{aligned}
\end{equation}
\begin{equation}\label{3.2b}
\begin{aligned}
\int_{\mathbb{R}^N}gu(t) \mathrm{d} x & \leq  \frac{1}{2\mu} \| g \|^{2} +\frac{\mu}{2}\|u(t)\|^{2}
\end{aligned}
\end{equation}
and
\begin{equation}\label{3.2}
\begin{aligned}
\int_{\mathbb{R}^N} f(u_t) u(t) \mathrm{d} x & \leq \frac{1}{2} \|f(u_t)\|^{2} +\frac{1}{2}\|u(t)\|^2 \\
& \leq \frac{1}{2}L_{f}^{2} \| u_t \|_{\mathcal{C}}^{2}+\frac{1}{2}\|u(t)\|^2.
\end{aligned}
\end{equation}
Therefore, by incorporating  \eqref{3.2a}-\eqref{3.2} into \eqref{3.1}, we are led to
\begin{equation}\label{3.4}
\begin{aligned}
\frac{\mathrm{d}\|u(t)\|^2}{\mathrm{d} t} & \leq -2\|\nabla u(t)\|^2+(\sigma-\mu+1) \|u(t)\|^2+ (\sigma+L_{f}^{2})\| u_t \|_{\mathcal{C}}^{2}+\frac{1}{2\mu} \| g \|^{2}\\
& \leq (\sigma-\mu+1)\|u(t)\|^2+  (\sigma+L_{f}^{2})\| u_t \|_{\mathcal{C}}^{2}+\frac{1}{\mu} \| g \|^{2}.
\end{aligned}
\end{equation}
Thanks to the  Gr{o}nwall inequality, we find, for all $t \geq 0$,
\begin{equation}\label{3.5}
\begin{aligned}
\|u(t)\|^{2} \leq e^{(\sigma-\mu+1) t}\|\phi(0)\|^{2}+  (\sigma+L_{f}^{2})\int_{0}^{t} e^{(\sigma-\mu+1) (t-s)}\| u_s \|_{\mathcal{C}}^{2}\mathrm{d}s+c_1,
\end{aligned}
\end{equation}
where $c_1=\frac{ \| g \|^{2}}{\mu(\mu-\sigma-1)}$. This implies, for all  $t\in[0, \infty)$ and $\zeta \in[-\tau, 0]$,
\begin{equation}\label{3.6}
\begin{aligned}
\|u(t+\zeta)\|^{2}& \leq e^{(\sigma-\mu+1) (t+\zeta)}\|\phi(0)\|^{2}+  (\sigma+L_{f}^{2})\int_{0}^{t+\zeta} e^{(\sigma-\mu+1) (t+\zeta-s)}\| u_s \|_{\mathcal{C}}^{2}\mathrm{d}s+c_1\\
& \leq e^{(\sigma-\mu+1) (t-\tau)}\|\phi(0)\|^{2}+  (\sigma+L_{f}^{2})e^{(\mu-\sigma-1)\tau}\int_{0}^{t} e^{(\sigma-\mu+1) (t-s)}\| u_s \|_{\mathcal{C}}^{2}\mathrm{d}s+c_1.
\end{aligned}
\end{equation}
Keeping in mind that $\left\|u_{t}\right\|_{\mathcal{C}}=\sup \left\{\|u(t+\xi)\|: \xi \in[-\tau, 0]\right\},$ we obtain
\begin{equation}\label{3.6a}
\begin{aligned}
\|u_t\|_{\mathcal{C}}^{2}& \leq e^{(\sigma-\mu+1) (t-\tau)}\|\phi(0)\|^{2}+  (\sigma+L_{f}^{2})e^{(\mu-\sigma-1)\tau}\int_{0}^{t} e^{(\sigma-\mu+1) (t-s)}\| u_s \|_{\mathcal{C}}^{2}\mathrm{d}s+c_1.
\end{aligned}
\end{equation}
Let $T_{1}\geq T_{\mathcal{D}}+\tau$ and replace $t$ by $T_{1}$ in \eqref{3.6a}. We then have
\begin{equation}\label{2.14}
\begin{aligned}
\|u_{T_{1}}\|_{\mathcal{C}}^{2}& \leq e^{(\sigma-\mu+1) (T_{1}-\tau)}\|\phi(0)\|^{2}+  (\sigma+L_{f}^{2})e^{(\mu-\sigma-1)\tau}\int_{0}^{T_{1}} e^{(\sigma-\mu+1) (T_{1}-s)}\| u_s \|_{\mathcal{C}}^{2}\mathrm{d}s+c_1.
\end{aligned}
\end{equation}
For every $T_{1} \geq T_{\mathcal{D}}$ and all $t \geq T_{1}$, multiplying \eqref{2.14} by $e^{(\mu-\sigma-1)\left(T_{1}-t\right)}$ implies
\begin{equation}\label{2.15}
\begin{aligned}
e^{(\mu-\sigma-1)\left(T_{1}-t\right)}\|u_{T_{1}}\|_{\mathcal{C}}^{2}& \leq e^{(\sigma-\mu+1) (t-\tau)}\|\phi(0)\|^{2}+  (\sigma+L_{f}^{2})e^{(\mu-\sigma-1)\tau}\int_{0}^{T_{1}} e^{(\sigma-\mu+1) (t-s)}\| u_s \|_{\mathcal{C}}^{2}\mathrm{d}s+c_1.
\end{aligned}
\end{equation}
Applying  Gr{o}nwall's inequality to \eqref{3.4} on the interval $\left[T_{1}+\zeta, t+\zeta\right], \zeta \in[-\tau, 0]$ for $t \geq T_{1},$ we obtain
\begin{equation}\label{2.16}
\begin{aligned}
&\|u(t+\zeta)\|^{2}+2\int_{T_{1}+\zeta}^{t+\zeta} e^{(\mu-\sigma-1) (s-t-\zeta)}\|\nabla u(s)\|^{2} \mathrm{d} s\\
\leq & e^{(\mu-\sigma-1)\left(T_{1}-t\right)}\left\|u\left(T_{1}+\zeta\right)\right\|^{2}+(\sigma+L_{f}^{2})\int_{T_{1}+\zeta}^{t+\zeta} e^{(\mu-\sigma-1)(s-t-\zeta)}\|u(s-\tau)\|^{2} \mathrm{d} s+\frac{c_1}{\mu-\sigma-1}e^{(\mu-\sigma-1)\tau}.
\end{aligned}
\end{equation}
By a change of variable in the second term on the left hand side of \eqref{2.16} and keeping in mind that $\left\|u_{t}\right\|_{\mathcal{C}}=\sup \left\{\|u(t+\xi)\|: \xi \in[-\tau, 0]\right\}$,  we deduce
\begin{equation}\label{2.17}
\begin{aligned}
&\|u_t\|_\mathcal{C}^{2}+2\int_{T_{1}}^{t} e^{(\mu-\sigma-1) (s-t)}\|u_s\|_{\mathcal{C}^1}^{2} \mathrm{d} s\\
\leq & e^{(\mu-\sigma-1)\left(T_{1}-t\right)}\left\|u_{T_{1}}\right\|_\mathcal{C}^{2}+(\sigma+L_{f}^{2})\int_{T_{1}-\tau}^{t} e^{(\mu-\sigma-1)(s-t-\tau)}\|u_s\|_\mathcal{C}^{2} \mathrm{d} s+\frac{c_1}{\mu-\sigma-1}e^{(\mu-\sigma-1)\tau}.
\end{aligned}
\end{equation}
It follows from Lemma \ref{lem2.2} that
\begin{equation}\label{2.18}
\begin{aligned}
(\sigma+L_{f}^{2})\int_{T_{1}-\tau}^{t} e^{(\mu-\sigma-1)(s-t-\tau)}\|u_s\|_\mathcal{C}^{2} \mathrm{d} s
\leq & (\sigma+L_{f}^{2})c_2^2\int_{T_{1}-\tau}^{t} e^{(\mu-\sigma-1)(s-t-\tau)}\mathrm{d} s\\
 \leq & (\sigma+L_{f}^{2})\frac{c_3^2}{\mu-\sigma-1},
\end{aligned}
\end{equation}
with $c_3=2(\frac{M}{\mu}+\frac{M\beta}{\mu(\mu-\beta)})$ and from \eqref{2.15}, \eqref{2.17} and \eqref{2.18} we derive that
\begin{equation}\label{2.19}
\begin{aligned}
&\int_{T_{1}}^{t} e^{(\mu-\sigma-1) (s-t)}\|u_s\|_{\mathcal{C}^1}^{2} \mathrm{d} s
\leq & \frac{1}{2}e^{(\sigma-\mu+1) (t-\tau)}\|\phi(0)\|^{2}+(\sigma+L_{f}^{2})\frac{c_3^2}{\mu-\sigma-1}+\frac{c_1}{2}.
\end{aligned}
\end{equation}
By replacing $t$ by $t+1$ and $T_{1}$ by $t $ in \eqref{2.19} for $s \in[t , t+1]$ and $t \geq T_{\mathcal{D}}$, we have
\begin{equation}\label{2.20}
\begin{aligned}
&\int_{t}^{t+1} e^{(\mu-\sigma-1) (s-t)}\|u_s\|_{\mathcal{C}^1}^{2} \mathrm{d} s
\leq & \frac{1}{2}e^{(\mu-\sigma-1)\tau }\|\phi(0)\|^{2}+(\sigma+L_{f}^{2})\frac{c_3^2}{\mu-\sigma-1}+\frac{c_1}{2},
\end{aligned}
\end{equation}
indicating
\begin{equation}\label{2.20a}
\begin{aligned}
&\int_{t}^{t+1} \|u_s\|_{\mathcal{C}^1}^{2} \mathrm{d} s
\leq & \frac{1}{2}e^{(\mu-\sigma-1)\tau }\|\phi(0)\|^{2}+(\sigma+L_{f}^{2})\frac{c_3^2}{\mu-\sigma-1}+\frac{c_1}{2}.
\end{aligned}
\end{equation}
 This completes the proof.
\end{proof}

The following result provides a uniform estimate for $u$ in $\mathcal{C}^1$.

\begin{lem}\label{lem2.4}
Assume that $\mathbf{Hypothesis \  A1}$ and assumptions of Lemma \ref{lem2.2} hold and $T_{\mathcal{D}}$ is defined in Lemma \ref{lem2.2}. Then, for all $t \geq T_{\mathcal{D}}$, we have
\begin{equation}\label{2.25}
\begin{aligned}
\left\|u_{t}^\phi\left(\cdot \right)\right\|_{\mathcal{C}^1}^{2}\leq c_4+2(\sigma^2+L_{f}^{2})c_3^2+2 \|g\|^2,
\end{aligned}
\end{equation}
where $c_3$ and $c_4$ are defined in Lemma \ref{lem2.3}.
\end{lem}
\begin{proof} Taking the inner product of \eqref{1} with $-\Delta v$ in $L^{2}(\mathbb{R}^N),$
\begin{equation}\label{2.22}
\begin{aligned}
\frac{1}{2} \frac{\mathrm{d}}{\mathrm{d} t}\|\nabla u\|^{2}=-\|\Delta u\|^{2}-\mu \|\nabla u\|^{2}-\int_{\mathbb{R}^N} \sigma u(t-\tau) \Delta u(t) \mathrm{d}x -\int_{\mathbb{R}^N} f(u(t-\tau)) \Delta u(t) \mathrm{d}x-\int_{\mathbb{R}^N} g \Delta u(t) \mathrm{d}x.
\end{aligned}
\end{equation}
In the sequel,  we estimate each term on the right-hand side of \eqref{2.22}. Apparently, by the elementary inequality, we have
\begin{equation}\label{2.23}
\begin{aligned}
\left|-\sigma\int_{\mathbb{R}^N}  u(t-\tau) \Delta u\mathrm{d} x\right| & \leq  \sigma \left(\sigma \| u(t-\tau)\|^{2}+\frac{1}{4\sigma}\|\Delta u\|^{2}\right),
\end{aligned}
\end{equation}
\begin{equation}\label{2.24}
\begin{aligned}
 \int_{\mathbb{R}^N}  f(u(t-\tau)) \Delta u \mathrm{d} x & \leq   \frac{1}{2}(\left( \|f(u(t-\tau))\|^{2}+\|\Delta u\|^{2}\right) ) \\
& \leq     \frac{1}{2}(L_{f}^{2}\|u(t-\tau)\|^{2}+\|\Delta u\|^{2}),
\end{aligned}
\end{equation}
and
\begin{equation}\label{2.25}
\begin{aligned}
\left|-\int_{\mathbb{R}^N} g \Delta u(t) \mathrm{d}x\right| \leq \left\|g\right\|^{2}+\frac{1}{4}\|\Delta u\|^{2}.
\end{aligned}
\end{equation}
Consequently, it follows from \eqref{2.22}-\eqref{2.25} that
\begin{equation}\label{2.26}
\begin{aligned}
\frac{\mathrm{d}}{\mathrm{d} t}\|\nabla u\|^{2} \leq  2(\sigma^2+L_{f}^{2}) \|u(t-\tau)\|^{2}+2 \|g\|^2.
\end{aligned}
\end{equation}
Let $T_{\mathcal{D}}$ be the positive constant (absorbing time) obtained in Lemma \ref{lem2.2}. Taking $t \geq T_{\mathcal{D}}+\tau$ and $\rho \in(t+\zeta, s+\zeta), \zeta \in[-\tau, 0], s \in[t, t+1],$ and then integrating \eqref{2.26} over $(\rho, s+\zeta),$ we obtain
\begin{equation}\label{2.27}
\begin{aligned}
\|u(s+\zeta)\|_{\mathbb{X}^1}^{2} \leq &\|u(\rho)\|_{\mathbb{X}^1}^{2} +2(\sigma^2+L_{f}^{2})\int_{t+\zeta}^{t+\zeta+1}\|u(s-\tau)\|^{2} \mathrm{d} s+2 \|g\|^2.
\end{aligned}
\end{equation}
Integrating the above inequality with respect to $\rho$ over $(t+\zeta, t+\zeta+1)$, we have
\begin{equation}\label{2.28}
\begin{aligned}
\|u(s+\zeta)\|_{\mathbb{X}^1}^{2} \leq &\int_{t+\zeta}^{t+\zeta+1}\|u(\rho)\|_{\mathbb{X}^1}^{2}\mathrm{d} \rho +2(\sigma^2+L_{f}^{2})\int_{t+\zeta}^{t+\zeta+1}\|u(s-\tau)\|^{2} \mathrm{d} s+2 \|g\|^2.
\end{aligned}
\end{equation}
and hence it follows from lemmas \ref{lem2.2} and \ref{lem2.3} that
\begin{equation}\label{2.29}
\begin{aligned}
\|u_s\|_{\mathcal{C}^1}^{2} \leq &\int_{t}^{t+1}\|u_s\|_{\mathcal{C}^1}^{2}\mathrm{d} s
+2(\sigma^2+L_{f}^{2})\int_{t-\tau}^{t+1}\|u_s\|^{2} \mathrm{d} s+2 \|g\|^2\\
\leq &c_4+2(\sigma^2+L_{f}^{2})c_3^2+2 \|g\|^2.
\end{aligned}
\end{equation}
This completes the proof.
\end{proof}

To show the asymptotic compactness of the infinite dimensional dynamical system $\Phi$,  we need  the  following uniform a priori estimates for far-field values of solutions. The proof follows the idea of Lemma 5 in \cite{W99} and for  readers' convenience, we provide the complete details.

\begin{lem}\label{lem2.5} Assume that $\mathbf{Hypothesis\  A1}$  and assumptions of Lemma \ref{lem2.2} hold. Then, for every $\varepsilon>0$, there exist $ T(\varepsilon)$ and $R(\varepsilon)$ such that for  all $t \geq T(\varepsilon)$ and $K \geq R(\varepsilon)$, we have
$$
\sup_{\zeta\in [-\tau, 0]}\int_{x\in \Omega_K^C}|u(t+\zeta)|^2 \mathrm{~d} x \leq \varepsilon,
$$
where $T(\varepsilon)$ and $R(\varepsilon)$ depend on $\varepsilon$.
\end{lem}
\begin{proof} Similar to \cite{W99}, we introduce the following smooth function $0 \leq\chi(s)\leq 1, s \in \mathbb{R}^{+}$ such that
\begin{equation}\label{2.30}
\chi(s)=\left\{\begin{array}{l}0, 0 \leq s \leq 1,\\ 1,s \geq 2.\end{array}\right.
\end{equation}
It follows from the smoothness of $\chi(s)$ that  there exists a constant $C$ such that $\left|\chi^{\prime}(s)\right| \leq C$ for $s \in \mathbb{R}^{+}$.
Multiplying both sides of \eqref{1} by $\chi\left(\frac{|x|^2}{K^2}\right)u$ and integrating on $\mathbb{R}^N$ imply
\begin{equation}\label{2.31}
\begin{aligned}
\frac{1}{2} \frac{\mathrm{d}}{\mathrm{d} t} \int_{\mathbb{R}^N}\chi\left(\frac{|x|^2}{K^2}\right)|u(t)|^2dx&=\int_{\mathbb{R}^N} \chi\left(\frac{|x|^2}{K^2}\right)u \Delta u dx-\mu \int_{\mathbb{R}^N} \chi\left(\frac{|x|^2}{K^2}\right)|u(t)|^2dx \\
&+\int_{\mathbb{R}^N} \chi\left(\frac{|x|^2}{K^2}\right)\sigma u(t-\tau) u(t) \mathrm{d}x +\int_{\mathbb{R}^N} \chi\left(\frac{|x|^2}{K^2}\right) f(u(t-\tau)) u(t) \mathrm{d}x\\&+\int_{\mathbb{R}^N} \chi\left(\frac{|x|^2}{K^2}\right) g u(t) \mathrm{d}x.
\end{aligned}
\end{equation}
We estimate each term on the right hand side of \eqref{2.31} as follows. First, we have
\begin{equation}\label{2.32}
\begin{aligned}
\int_{\mathbb{R}^N} \chi\left(\frac{|x|^2}{K^2}\right) u(t) \Delta u(t) d x= &-\int_{\mathbb{R}^N} \chi\left(\frac{|x|^2}{K^2}\right)|\nabla u(t)|^2 d x-\int_{K \leq|x| \leq \sqrt{2} K} \frac{2 x}{K^2}\chi'\left(\frac{|x|^2}{K^2}\right) u(t) \nabla u(t) d x.
\end{aligned}
\end{equation}
It follows from the Young inequality that
\begin{equation}\label{2.33}
\begin{aligned}
\left|\int_{K \leq|x| \leq \sqrt{2} K} \frac{2x}{K^2} \chi^{\prime}\left(\frac{|x|^2}{K^2}\right) u(t) \nabla u(t) d x\right| & \leq \frac{2\sqrt{2}}{K} \int_{K \leq|x| \leq \sqrt{2} K} \chi'\left(\frac{|x|^2}{K^2}\right)|u(t) \| \nabla u(t)| d x\\
 & \leq \frac{\sqrt{2}C}{K}\left(\|u(t)\|^2+\|\nabla u(t)\|^2\right).
\end{aligned}
\end{equation}
By \eqref{2.32} and \eqref{2.33},
\begin{equation}\label{2.34}
\begin{aligned}
\int_{\mathbb{R}^N} \chi\left(\frac{|x|^2}{K^2}\right) u(t) \Delta u(t) d x\leq &-\int_{\mathbb{R}^N} \chi\left(\frac{|x|^2}{K^2}\right)|\nabla u(t)|^2 d x\\
&+\frac{\sqrt{2}C}{K}\left(\|u(t)\|^2+\|\nabla u(t)\|^2\right).
\end{aligned}
\end{equation}
Next, by the Young inequality once more, we have
\begin{equation}\label{2.35}
\begin{aligned}
 \sigma\int_{\mathbb{R}^N}  \chi\left(\frac{|x|^2}{K^2}\right) u(t-\tau) u(t) \mathrm{d} x & \leq  \sigma \int_{\mathbb{R}^N}  \chi\left(\frac{|x|^2}{K^2}\right)[\frac{1}{2} |u(t-\tau)|^2+\frac{1}{2}|u(t)|^2 ]\mathrm{d} x\\
 & \leq  \frac{\sigma}{2}  \int_{\mathbb{R}^N}  \chi\left(\frac{|x|^2}{K^2}\right)|u(t-\tau)|^2\mathrm{d}x\\
 &\quad+\frac{\sigma}{2}\int_{\mathbb{R}^N}  \chi\left(\frac{|x|^2}{K^2}\right)|u(t)|^2\mathrm{d}x,
\end{aligned}
\end{equation}
\begin{equation}\label{2.36}
\begin{aligned}
 \int_{\mathbb{R}^N}\chi\left(\frac{|x|^2}{K^2}\right) f(u(t-\tau)) u(t)\mathrm{d} x & \leq   \frac{1}{2}\int_{\mathbb{R}^N}  \chi\left(\frac{|x|^2}{K^2}\right) |f(u(t-\tau))|^{2}\mathrm{d} x \\
 &\quad+\frac{1}{2}\int_{\mathbb{R}^N}  \chi\left(\frac{|x|^2}{K^2}\right) |u(t)|^{2}\mathrm{d} x  \\
& \leq  \int_{\mathbb{R}^N}  \chi\left(\frac{|x|^2}{K^2}\right) (L_{f}^{2}|u(t-\tau)|^{2}+|f(0)|^2)\mathrm{d} x \\
&\quad+\frac{1}{2}\int_{\mathbb{R}^N}  \chi\left(\frac{|x|^2}{K^2}\right) |u(t)|^{2}\mathrm{d} x,
\end{aligned}
\end{equation}
and
\begin{equation}\label{2.37}
\begin{aligned}
\int_{\mathbb{R}^N} \chi\left(\frac{|x|^2}{K^2}\right) g u(t) \mathrm{d}x \leq \frac{1}{2\mu}\int_{\mathbb{R}^N}  \chi\left(\frac{|x|^2}{K^2}\right) |g|^{2}\mathrm{d} x +\frac{\mu}{2}\int_{\mathbb{R}^N}  \chi\left(\frac{|x|^2}{K^2}\right) |u(t)|^{2}\mathrm{d} x.
\end{aligned}
\end{equation}
Incorporating  \eqref{2.34} to \eqref{2.37} into \eqref{2.31} we deduce
\begin{equation}\label{2.38}
\begin{aligned}
&\frac{\mathrm{d}}{\mathrm{d} t} \int_{\mathbb{R}^N}\chi\left(\frac{|x|^2}{K^2}\right)|u(t)|^{2}dx\\
&\quad\leq (1+\sigma-\mu)\int_{\mathbb{R}^N} \chi\left(\frac{|x|^2}{K^2}\right)|u(t)|^{2}dx-2\int_{\mathbb{R}^N} \chi\left(\frac{|x|^2}{K^2}\right)|\nabla u(t)|^2 d x\\
&\qquad+( \sigma +2L_f^2)  \int_{\mathbb{R}^N}  \chi\left(\frac{|x|^2}{K^2}\right)|u(t-\tau)|^2\mathrm{d}x+\frac{2\sqrt{2}C}{K}\left(\|u(t)\|^2+\|\nabla u(t)\|^2\right) \\&\qquad+\int_{\mathbb{R}^N}  \chi\left(\frac{|x|^2}{K^2}\right) (2|f(0)|^2+\frac{1}{\mu}|g|^{2})\mathrm{d} x.
\end{aligned}
\end{equation}
 Let $T_{\mathcal{D}}$ be the positive constant obtained in Lemma \ref{lem2.2} and $T_2\geq T_{\mathcal{D}}$. Thanks to  Gr{o}nwall's inequality on $[T_2, t]$, we derive
 \begin{equation}\label{2.39}
\begin{aligned}
&\int_{\mathbb{R}^N}\chi\left(\frac{|x|^2}{K^2}\right)|u(t)|^{2}\mathrm{d} x\\
&\quad\leq e^{(\sigma-\mu+1) (t-T_2)}\int_{\mathbb{R}^N}\chi\left(\frac{|x|^2}{K^2}\right)|u(T_2)|^2\mathrm{d} x+( \sigma +2L_f^2) \int_{T_2}^te^{(\sigma-\mu+1) (t-s)}\times\\
&\qquad\times \int_{\mathbb{R}^N}  \chi\left(\frac{|x|^2}{K^2}\right)|u(s-\tau)|^2\mathrm{d}x\mathrm{d}s+\int_{T_2}^te^{(\sigma-\mu+1) (t-s)}\int_{\mathbb{R}^N}  \chi\left(\frac{|x|^2}{K^2}\right) (2|f(0)|^2+\\&\qquad+\frac{1}{\mu}|g|^{2})\mathrm{d} x \mathrm{d}s +\int_{T_2}^te^{(\sigma-\mu+1) (t-s)}\frac{2\sqrt{2}C}{K}\left(\|u(s)\|^2+\|\nabla u(s)\|^2\right) \mathrm{d}s.
\end{aligned}
\end{equation}
Hence, for any $\zeta \in[-\tau, 0]$, we have
 \begin{equation}\label{2.39a}
\begin{aligned}
&\int_{\mathbb{R}^N}\chi\left(\frac{|x|^2}{K^2}\right)|u(t+\zeta)|^{2}\mathrm{d} x\\
&\quad\leq e^{(\sigma-\mu+1) (t+\zeta-T_2)}\int_{\mathbb{R}^N}\chi\left(\frac{|x|^2}{K^2}\right)|u(T_2)|^2\mathrm{d} x+(\sigma +2L_f^2) \int_{T_2}^{t+\zeta}e^{(\sigma-\mu+1) ((t+\zeta)-s)}\times\\
&\qquad\times \int_{\mathbb{R}^N}  \chi\left(\frac{|x|^2}{K^2}\right)|u(s-\tau)|^2\mathrm{d}x\mathrm{d}s+\int_{T_2}^{t+\zeta}e^{(\sigma-\mu+1) ((t+\zeta)-s)}\int_{\mathbb{R}^N}  \chi\left(\frac{|x|^2}{K^2}\right) (2|f(0)|^2+\\&\qquad+\frac{1}{\mu}|g|^{2})\mathrm{d} x \mathrm{d}s +\int_{T_2}^{t+\zeta}e^{(\sigma-\mu+1) ((t+\zeta)-s)}\frac{2\sqrt{2}C}{K}\left(\|u(s)\|^2+\|\nabla u(s)\|^2\right) \mathrm{d}s\\
&\quad\leq c_2 [e^{(\sigma-\mu+1) (t-T_2)}\int_{\mathbb{R}^N}\chi\left(\frac{|x|^2}{K^2}\right)|u(T_2)|^2\mathrm{d} x+( \sigma +2L_f^2) \int_{T_2}^te^{(\sigma-\mu+1) (t-s)}\times\\
& \qquad\times\int_{\mathbb{R}^N}  \chi\left(\frac{|x|^2}{K^2}\right)|u(s-\tau)|^2\mathrm{d}x\mathrm{d}s+\int_{T_2}^te^{(\sigma-\mu+1) (t-s)}\int_{\mathbb{R}^N}  \chi\left(\frac{|x|^2}{K^2}\right) (2|f(0)|^2+\\&\qquad+\frac{1}{\mu}|g|^{2})\mathrm{d} x \mathrm{d}s +\int_{T_2}^te^{(\sigma-\mu+1) (t-s)}\frac{2\sqrt{2}C}{K}\left(\|u(s)\|^2+\|\nabla u(s)\|^2\right) \mathrm{d}s].
\end{aligned}
\end{equation}
Now we estimate each term on the right hand side of \eqref{2.39a}. It follows from Lemma \ref{lem2.2} that
 \begin{equation}\label{2.40}
\begin{aligned}
c_2e^{-(\sigma-\mu+1) (t-T_2)}\int_{\mathbb{R}^N}\chi\left(\frac{|x|^2}{K^2}\right)|u(T_2)|^2dx&\leq c_2c_4e^{-(\sigma-\mu+1) (t-T_2)}
\end{aligned}
\end{equation}
and
 \begin{equation}\label{2.41}
\begin{aligned}
c_2( \sigma +2L_f^2) \int_{T_2}^te^{(\sigma-\mu+1) (s-t)}\int_{\mathbb{R}^N}  \chi\left(\frac{|x|^2}{K^2}\right)|u(s-\tau)|^2\mathrm{d}x\mathrm{d}s&\leq c_3c_4\int_{T_2}^te^{(\sigma-\mu+1) (s-t)}\mathrm{d}s,
\end{aligned}
\end{equation}
where $c_4=2\frac{c_1(\mu-\sigma-1)}{\mu-\sigma-1-c_3}$.
Therefore,  there exists $T_3(\varepsilon)>T_2$ such that for all $t\geq T_3(\varepsilon)$, we have
 \begin{equation}\label{2.42}
\begin{aligned}
c_2e^{-(\sigma-\mu+1) (t-T_2)}\int_{\mathbb{R}^N}\chi\left(\frac{|x|^2}{K^2}\right)|u^2(T_2)|dx&\leq \frac{\varepsilon}{4}.
\end{aligned}
\end{equation}
and
 \begin{equation}\label{2.43}
\begin{aligned}
c_2(\sigma +2L_f^2) \int_{T_2}^te^{(\sigma-\mu+1) (s-t)}\int_{\mathbb{R}^N}  \chi\left(\frac{|x|^2}{K^2}\right)|u(s-\tau)|^2\mathrm{d}x\mathrm{d}s&\leq \frac{\varepsilon}{4}.
\end{aligned}
\end{equation}
Since $g\in \mathbb{X}$ and $f(0)\in \mathbb{X}$, there exists $R_1(\varepsilon)$ such that  for any $K>R_1(\varepsilon)$
 \begin{equation}\label{2.44}
\begin{aligned}
c_2\int_{x\in \Omega_K^C}  \chi\left(\frac{|x|^2}{K^2}\right) (2|f(0)|^2+\frac{1}{\mu}|g|^{2})\mathrm{d} x &\leq \frac{\varepsilon}{4}.
\end{aligned}
\end{equation}
Hence, there exists $T_4(\varepsilon)>T_2$ such that for all $t\geq T_4$ and $K>R_1(\varepsilon)$
 \begin{equation}\label{2.45}
\begin{aligned}
c_2\int_{T_2}^te^{(\sigma-\mu+1) (s-t)}\int_{\mathbb{R}^N}  \chi\left(\frac{|x|^2}{K^2}\right) (2|f(0)|^2+\frac{1}{\mu}|g|^{2})\mathrm{d} x \mathrm{d}s
&\leq c_2\int_{T_2}^te^{(\sigma-\mu+1) (s-t)}\frac{\varepsilon}{4} \mathrm{d}s\\
&\leq \frac{\varepsilon}{4}.
\end{aligned}
\end{equation}
At last, it follows from lemmas \ref{lem2.2} and \ref{lem2.4} that $\|u(s)\|^2+\|\nabla u(s)\|^2\leq \|u_s\|_\mathcal{C}^2+\|u(s)\|_{\mathcal{C}^1}^2$ and hence there exists $T_5(\varepsilon)>T_2$ and $R_2(\varepsilon)>0$ such that for all $t\geq T_5$ and $K>R_2(\varepsilon)$, we have
 \begin{equation}\label{2.46}
\begin{aligned}
c_2\int_{T_2}^te^{(\sigma-\mu+1) (s-t)}\frac{2\sqrt{2}C}{K}\left(\|u(s)\|^2+\|\nabla u(s)\|^2\right) \mathrm{d}s
&\leq \frac{2\sqrt{2}C}{K}\int_{T_2}^te^{(\sigma-\mu+1) (s-t)}c_2(c_4+c_5)\mathrm{d}s\\
&\leq \frac{\varepsilon}{4}.
\end{aligned}
\end{equation}
Taking $T(\varepsilon)=\max\{T_2(\varepsilon), T_3(\varepsilon), T_4(\varepsilon), T_5(\varepsilon)\}$ and $R(\varepsilon)=\max\{R_1(\varepsilon), R_2(\varepsilon)\}$. It follows from  \eqref{2.42} to \eqref{2.46} that for any $t\geq T(\varepsilon)$, $\zeta\in[-\tau,0]$ and $R\geq R(\varepsilon)$, we have
 \begin{equation}\label{2.47}
\begin{aligned}
\sup_{\zeta\in[-\tau,0]} \int_{x\in \Omega_K^C}\chi\left(\frac{|x|^2}{K^2}\right)|u(t+\zeta)|^2\mathrm{d} x
&\leq \sup_{\zeta\in[-\tau,0]}\int_{\mathbb{R}^N}\chi\left(\frac{|x|^2}{K^2}\right)|u(t+\zeta)|^2\mathrm{d} x \\
&\leq  \varepsilon.
\end{aligned}
\end{equation}
The proof is completed.
\end{proof}

\section{Existence of global attractors}
In this section, we investigate the existence of global attractors of \eqref{1} by proving the asymptotic  compactness  of the dynamical system $\Phi$ together with the absorbing set obtained in Lemma \ref{lem2.1}. We first show the following results about the asymptotic  compactness of $\Phi(t)|_{\Omega_K}: \mathcal{C}_{\Omega_K}\rightarrow \mathcal{C}_{\Omega_K}$ defined  by
 \begin{equation}\label{4.1}
\Phi(t)|_{\Omega_K} \varphi =v_t^\varphi
\end{equation}
for all $\varphi\in \mathcal{C}_{\Omega_K}$ with $v_t^\varphi$ being solution to \eqref{2.3}.
\begin{lem}\label{lem4.1} Assume that $\mathbf{Hypothesis \  A1}$ and assumptions of Lemma \ref{lem2.2} hold and $T_{\mathcal{D}}$ is defined in Lemma \ref{lem2.2}. Then, the infinite dimensional dynamical system $\Phi|_{\Omega_K}$ defined by \eqref{4.1} is asymptotically compact in $\mathcal{C}_{\Omega_K}$. That is,  for any bounded $\mathcal{D}_{\Omega_K}\subseteq \mathcal{C}_{\Omega_K}$  and
$\{\varphi_n\}\subseteq \mathcal{D}_{\Omega_K}$, if $t_{n} \rightarrow \infty$, then the sequence $\{\Phi(t_n)|_{\Omega_K}\varphi_n\}_{n=1}^{\infty}$ has a convergent subsequence in $\mathcal{C}_{\Omega_K}$.
\end{lem}
\begin{proof}
It suffices to show that $\{\Phi(t_n)|_{\Omega_K}\varphi_n\}_{n=1}^{\infty}$ is precompact in $\mathcal{C}_{\Omega_K}$ thanks to the Ascoli-Arzel\`a theorem. Lemma \ref{lem2.2} implies that $\Phi(t_n, \varphi_n)$ is uniformly bounded in $\mathcal{C}$, which directly indicates that $\{\Phi(t_n)|_{\Omega_K}\varphi_n\}_{n=1}^{\infty}$ is uniformly bounded in  $\mathcal{C}_{\Omega_K}$.

In the sequel, we show that $\{\Phi(t_n)|_{\Omega_K}\varphi_n\}_{n=1}^{\infty}$ is uniformly equicontinuous in $\mathcal{C}_{\Omega_K}$.
It follows from lemmas \ref{lem2.2} and \ref{lem2.3} that
\begin{equation}\label{4.2b}
\begin{aligned}
\left\|u^\varphi_{s}\left(\cdot\right)\right\|_{\mathcal{C}_{\Omega_K}}\leq\left\|u^\phi_{s}\left(\cdot\right)\right\|_{\mathcal{C}} \leq   c_3
\end{aligned}
\end{equation}
and
\begin{equation}\label{4.2a}
\begin{aligned}
\int_{t-\tau}^{t+1}\left\|u^\phi_{s}\left(\cdot\right)\right\|_{\mathcal{C}_{\Omega_K}^{1}}  \mathrm{d} s\leq \int_{t-\tau}^{t+1}\left\|u^\phi_{s}\left(\cdot\right)\right\|_{\mathcal{C}^{1}} \mathrm{d} s \leq \sqrt{ c_4}.
\end{aligned}
\end{equation}
It follows from \eqref{2.3} that we can find large enough $t_n>T_{\mathcal{D}}+\tau$ such that
\begin{equation}\label{5.1}
\begin{aligned}
&\int_{t_{n}-\tau}^{t_{n}+1}\left\|\frac{d}{\mathrm{d} s} u^\varphi_s\right\|_{\mathbb{X}_{\Omega_K}} \mathrm{d} s\\ &\quad\leq  \int_{t_{n}-\tau}^{t_{n}+1}[\|u^\varphi_s\|_{\mathcal{C}^1_{\Omega_K}}+(\mu+\sigma+L_f)\|u^\varphi(s)\|_{\mathcal{C}_{\Omega_K}}+\|f(0)\|+\|g\|]\mathrm{d} s\\
&\quad\leq
 \sqrt{ c_4}+((\mu+\sigma+L_f) c_3+\|f(0)\|+\|g\|)(1+\tau)\triangleq c_5.
\end{aligned}
\end{equation}
Notice that $t_{n} \geq T_{\mathcal{D}}+\tau$  is equivalent to $n \geq N_{D}$ for some positive integer $N_{D}$. Therefore, for every $n \geq N_{D}$ and $s_{1}, s_{2} \in  [t_{n}-\tau, t_{n}+1]$, we have
\begin{equation}\label{5.2}
\begin{aligned}
&\left\|u^\varphi\left(s_{2}\right)-u^\varphi\left(s_{1}\right)\right\|
\leq  \left|s_{2}-s_{1}\right|  \int_{s_{1}}^{s_{2}}\left\|\frac{d}{\mathrm{d} s} u^\varphi(s)\right\|_{\mathbb{X}_{\Omega_K}} \mathrm{d} s
\leq  c_5\left|s_{2}-s_{1}\right|.
\end{aligned}
\end{equation}
which  implies the desired equicontinuity. For each fixed $\zeta \in [-\tau,0]$, by Lemma \ref{lem2.4} we can see
$\{\Phi(t_n)|_{\Omega_K}\varphi_n\}_{n=1}^{\infty}$ is bounded in $\mathcal{C}_{\Omega_K}^1$. Then it follows from the compact embedding of $\mathbb{X}^1_{\Omega_K}\hookrightarrow \mathbb{X}_{\Omega_K}$ that $\{\Phi(t_n)|_{\Omega_K}\varphi_n\}_{n=1}^{\infty}$ is compact in $\mathcal{C}_{\Omega_K}$. This completes the proof.
\end{proof}

We are now in a position to present our main result, i.e., the existence of a global attractor for $\phi$ in $\mathcal{C}$.

\begin{thm}\label{thm4.1} Assume that $\mathbf{Hypothesis \  A1}$ and assumptions of Lemma \ref{lem2.2} hold and $T_{\mathcal{D}}$ is defined in Lemma \ref{lem2.2}. Then, \eqref{1} admits a global attractor $\mathcal{A}$ which is a compact invariant set and attracts every bounded set in $\mathcal{C}$.
\end{thm}
\begin{proof}
Let $t_{n} \rightarrow \infty$ and take  $\{\phi_n\}_{n=1}^{\infty}\subseteq  \mathcal{D}$. It follows from Lemma \ref{lem2.1} that $\{\Phi\left(t_{n}, \phi_n\right)\}_{n=1}^{\infty}$ is bounded in $\mathcal{C}$. Therefore, there is $\xi \in \mathcal{C}$ such that, up to a subsequence,
 \begin{equation}\label{2.48}
\begin{aligned}
\Phi\left(t_{n}, \phi_n\right) \rightarrow \xi \quad \text { weakly in } \mathcal{C}.
\end{aligned}
\end{equation}
Next, we prove the above weak convergence  is actually strong convergence. For any given $\varepsilon>0$, by Lemma \ref{lem2.5}, there exist $T (\varepsilon)$ and $R (\varepsilon)$ such that, for all $t \geqslant T (\varepsilon)$,
 \begin{equation}\label{2.49}
\begin{aligned}
\int_{|x| \geqslant R (\varepsilon)}\left|\Phi\left(t, \phi_n\right) \right|^2 \mathrm{d} x \leqslant \varepsilon.
\end{aligned}
\end{equation}
Since $t_n \rightarrow \infty$, there is $N_1=N_1(\varepsilon)$ such that $t_n \geqslant T(\varepsilon)$ for every $n \geqslant N_1$. Hence, it follows from \eqref{2.49} that, for all $n \geqslant N_1$,
 \begin{equation}\label{2.50}
\begin{aligned}
\int_{|x| \geqslant R (\varepsilon)}\left|\Phi\left(t_{n}, \phi_n\right)\right|^2 \mathrm{d} x \leqslant \varepsilon.
\end{aligned}
\end{equation}
On the other hand, it follows from Lemma \ref{lem2.4} that   for all $t \geqslant T_\mathcal{D}$,
 \begin{equation}\label{2.51}
\begin{aligned}
\left\|\Phi\left(t, \phi_n\right)\right\|_{\mathcal{C}^1}^2 \leqslant c_5.
\end{aligned}
\end{equation}
Let $N_2=N_2(\varepsilon)$ be large enough such that $t_n \geqslant T_\mathcal{D}$ for $n \geqslant N_2$. Then, by \eqref{2.51} we find that, for all $n \geqslant N_2$,
 \begin{equation}\label{2.52}
\begin{aligned}
\left\|\Phi(t_{n})|_{\Omega_R (\varepsilon)} (\chi_{{\Omega_R (\varepsilon)}}\phi_n) \right\|_{\mathcal{C}_{\Omega_R (\varepsilon)}^1}^2 \leqslant c_5,
\end{aligned}
\end{equation}
where $\chi_{{\Omega_R (\varepsilon)}}$ is the characteristic function defined in \eqref{2.1}.
By  Lemma  \ref{lem4.1}, we can see
$$
\Phi(t_{n})|_{\Omega_R (\varepsilon)} (\chi_{{\Omega_R (\varepsilon)}}\phi_n)  \rightarrow \xi \quad \text { strongly in } \mathcal{C}_{\Omega_R (\varepsilon)},
$$
which shows that for the given $\varepsilon>0$, there exists $N_3=N_3(\varepsilon)$ such that, for all $n \geqslant N_3$,
 \begin{equation}\label{2.53}
\begin{aligned}
\left\|\Phi(t_{n})|_{\Omega_R (\varepsilon)} (\chi_{{\Omega_R (\varepsilon)}}\phi_n) -\xi\right\|_{\mathcal{C}_{\Omega_R (\varepsilon)}}^2 \leqslant \varepsilon.
\end{aligned}
\end{equation}
Recall that $\xi \in \mathbb{X}$. Therefore, there exists $R_*=R_*(\varepsilon)$ such that
 \begin{equation}\label{2.54}
\begin{aligned}
\int_{|x| \geqslant R_*}|\xi(x)|^2 d x \leqslant \varepsilon.
\end{aligned}
\end{equation}
Let $ R =\max \left\{R(\varepsilon), R_*\right\}$ and $N_4=\max \left\{N_1, N_2, N_3\right\}$. By \eqref{2.50}, \eqref{2.53} and \eqref{2.54}, we find that for all $n \geqslant N_4$,
 \begin{equation}\label{2.55}
\begin{aligned}
\left\|\Phi\left(t_{n}, \phi_n\right)-\xi\right\|_{\mathcal{C}}^2 \leqslant & \int_{|x| \leqslant  R }\left|\Phi(t_{n})|_{\Omega_R} (\chi_{{\Omega_R}}\phi_n)-\xi\right|^2 d x+\int_{|x| \geqslant  R}\left|\Phi\left(t_{n}, \phi_n\right)-\xi\right|^2 d x
 \\
\leqslant &  3\varepsilon,
\end{aligned}
\end{equation}
which shows that
$$
\Phi\left(t_{n}, \phi_n\right) \rightarrow \xi \quad \text { strongly in } \mathcal{C},
$$
implying the asymptotic compactness of $\Phi$ in $\mathcal{C}$.  Therefore, it follows from Lemmas \ref{lem2.1} and \ref{lem2.2} that $\Phi$ admits a global attractor.
\end{proof}
\section{Squeezing property}
This section is devoted to the squeezing property of the solutions to \eqref{2.3} and \eqref{2.4}.  In the remaining part of this paper, we  always assume that conditions of Theorem \ref{thm4.1} are satisfied and $T(\varepsilon)$ as well as $R(\varepsilon)$ are defined in Lemma \ref{lem2.5}. We will set $K=R(\varepsilon)$.

We first consider the following linear part of \eqref{2.3} on $\mathbb{X}_{\Omega_K}$.
\begin{equation}\label{5.1}
\left\{\begin{array}{l}
\frac{\partial \tilde{v}(x,t)}{\partial t}=\Delta \tilde{v}(x,t)-\mu \tilde{v}(x,t)+\sigma \tilde{v}(x,t-\tau), \\
\tilde{v}(x,s)=\phi(x, s)\chi_{\Omega_K}(x), s \in[-\tau, 0], x \in  \Omega_K,\\
\tilde{v}(x,t)=0, t \in(0, \infty), x \in \partial \Omega_K.
\end{array}\right.
\end{equation}
In order to set \eqref{5.1} in the abstract semigroup framework, we define $A: \mathbb{X}_{\Omega_K}\rightarrow \mathbb{X}_{\Omega_K}$ by
 \begin{equation}\label{5.2}
\begin{aligned}
Ax=\ddot{x}
\end{aligned}
\end{equation}
with domain $\operatorname{Dom}\left(A\right)=\left\{x \in C^2(\Omega_K); x(\partial \Omega_K)=0\right\}$, $L: \mathcal{C}_{\Omega_K}\rightarrow \mathbb{X}_{\Omega_K}$ by
 \begin{equation}\label{5.3}
\begin{aligned}
L\phi \triangleq -\mu\phi(0)+\sigma \phi(-\tau)
\end{aligned}
\end{equation}
for any $\phi\in \mathcal{C}_{\Omega_K}$.

It follows from \cite[Theorem 2.6]{WJ}  that   \eqref{5.1} admits a global solution   $\tilde{v}^\phi(\cdot):[-r, \infty] \rightarrow \mathbb{X}_{\Omega_K}$. Define the linear semigroup  $U(t): \mathcal{C}_{\Omega_K}\rightarrow \mathcal{C}_{\Omega_K}$ by $U(t)\phi=\tilde{v}^\phi_t(\cdot)$. Let  $A_U: \mathcal{C}_{\Omega_K}\rightarrow \mathcal{C}_{\Omega_K}$ be the infinitesimal generator of $U(t)$.

Consider the following eigenvalue problem on $\Omega_K$:
 \begin{equation}\label{3.9}
\begin{aligned}
-\Delta u(x)=\mu u(x),\left.\quad u(x)\right|_{x \in \partial \Omega_K}=0, \quad x \in \Omega_K,
\end{aligned}
\end{equation}
which has a family of solutions (eigenfunctions) $\left\{e_{m, K}\right\}_{m \in \mathbb{N}}$ with eigenvalues $\left\{\mu_{m, K}\right\}_{m \in \mathbb{N}}$ such that
$$
0<\mu_{1, K} \leq \mu_{2, K} \leq \cdots \leq \mu_{m, K} \leq \cdots, \quad \mu_{m, K} \rightarrow+\infty \quad \text { as } \quad m \rightarrow+\infty.
$$
Since $A_U$ is compact, it follows from \cite[Theorem 1.2 (i)]{WJ} that the spectrum of $A_U$  are point spectra, which we denote by $\varrho_1>\varrho_2>\cdots$ with multiplicity $n_1, n_2,\cdots$.
Moreover, it follows from  \cite{WJ} that the characteristic values  $\varrho_1>\varrho_2>\cdots$  of the linear part  $A_U$ are the roots of the following characteristic equation
 \begin{equation}\label{3.10}
\begin{aligned}
\mu^2_{m, K} -\left(\lambda+\mu-\sigma e^{-\lambda \tau}\right) =0, m=1,2, \cdots,
\end{aligned}
\end{equation}
where  $\varrho_1$ is the first eigenvalue of  $A_U$  defined as
 \begin{equation}\label{3.11}
\begin{aligned}
\varrho_1=\max \left\{\operatorname{Re} \lambda: \mu^2_{m, K} -\left(\lambda+\mu-\sigma e^{-\lambda \tau}\right) =0\right\}, n=1,2, \cdots.
\end{aligned}
\end{equation}

It follows from Theorem 1.10 on P71 in \cite{WJ} that there exists a number such that all eigenvalues of $A_U$ lie in its left, implying that there exists at most a finite number of  eigenvalues of $A_U$ which are positive.  Hence,   there exists $m\geq 1$ such that  $\varrho_m<0$, and there is a
 \begin{equation}\label{3.12}
\begin{aligned}
k_m=n_1+n_2+\cdots+n_m
\end{aligned}
\end{equation}
dimensional  subspace $\mathcal{C}_{\Omega_K}^U$ such that $\mathcal{C}_{\Omega_K}$ is decomposed by  $A_U$  as
 \begin{equation}\label{3.12a}
\begin{aligned}
\mathcal{C}_{\Omega_K}= \mathcal{C}_{\Omega_K} ^U \bigoplus  \mathcal{C}_{\Omega_K} ^S.
\end{aligned}
\end{equation}
 Let $P_{k_m}$ and $Q_{k_m}$ be the projection of $\mathcal{C}_{\Omega_K}$ onto $ \mathcal{C}_{\Omega_K}^U$ and $ \mathcal{C}_{\Omega_K}^S$ respectively, that is
   \begin{equation}\label{3.12b}
\begin{aligned}
\mathcal{C}_{\Omega_K}^U=P_{k_m}\mathcal{C}_{\Omega_K}
\end{aligned}
\end{equation}
and
 \begin{equation}\label{3.12c}
\begin{aligned}
\mathcal{C}_{\Omega_K}^S=(I-P_{k_m})\mathcal{C}_{\Omega_K}=Q_{k_m}\mathcal{C}_{\Omega_K}.
\end{aligned}
\end{equation}
It follows from the definition of $P_{k_m}$ and $Q_{k_m}$ that
\begin{equation}\label{3.13}
\begin{aligned}
\left\|U(t)Q_{k_m} x\right\| & \leq K_m e^{\varrho_m t}\|x\|, & & t \geq0,
\end{aligned}
\end{equation}
where $K_m$ and $\gamma$ are some positive constants.

To show the squeezing property, we extend the domain of $U(t)$ to the following space of some discontinuous functions
 \begin{equation}\label{3.14}
\begin{aligned}
\hat{C}=\left\{\phi:[-\tau, 0] \rightarrow \mathcal{C}_{\Omega_K} ; \phi|_{[-\tau, 0)} \quad \text {is continuous and } \lim _{\theta \rightarrow 0^{-}} \phi(\theta) \in \mathcal{C}_{\Omega_K} \quad \text{exists} \right\}
\end{aligned}
\end{equation}
and introduce the following informal variation of  constant formula established in \cite{WJ}
 \begin{equation}\label{5.13}
\begin{aligned}
v(t) & =U(t) \varphi+\int_0^t\left[U(t-s) X_0 f(v_s)\right](0) d s, \quad t \geq 0.
\end{aligned}
\end{equation}
It is proved by  \cite[Theorem 2.1]{WJ} that the function $v:[-\tau, \infty) \rightarrow \mathcal{C}_{\Omega_K}$ defined by \eqref{5.13} satisfies \eqref{2.3} with $v_0=\phi \chi_{\Omega_K} \triangleq \varphi$, where $X_0:[-\tau, 0] \rightarrow B(\mathcal{C}_{\Omega_K}, \mathcal{C}_{\Omega_K})$ is given by $X_0(\theta)=0$ if $-\tau \leq \theta<0$ and $X_0(0)=I d$.

\begin{rem}
Generally, the solution semigroup defined by \eqref{3.15} is not defined at discontinuous functions and the integral in the formula is undefined as an integral in the phase space. However,  if interpreted correctly, \eqref{3.15} does make sense. Details can be found in \cite[pages 144-145]{CM}.
\end{rem}

Let  $\Phi|_{\Omega_K}$ be defined  by \eqref{4.1} and define $\Phi|_{\Omega_K^C}: \mathcal{C} \rightarrow \mathcal{C}$  by
 \begin{equation}\label{4.2}
\Phi(t)|_{\Omega_K^C} \varphi =w_t^\varphi
\end{equation}
for all $\varphi\in \mathcal{C}$ with $w_t^\varphi$ being solution to \eqref{2.4}.
Therefore, for any $\varphi, \psi\in \mathcal{C}$, we can decompose $\|\Phi(t,\varphi)-\Phi(t,\psi)\|_{\mathcal{C}}$ into two parts
 \begin{equation}\label{5.14a}
\begin{aligned}
 \|\Phi(t,\varphi)-\Phi(t,\psi) \|_{\mathcal{C}}\leq\left\| \Phi(t)|_{\mathcal{C}_{\Omega_K}}\varphi- \Phi(t)|_{\mathcal{C}_{\Omega_K}}\psi\right\|_{\mathcal{C}_{\Omega_K}}+\left\| \Phi(t)|_{\mathcal{C}_{\Omega_K^C}}\varphi- \Phi(t)|_{\mathcal{C}_{\Omega_K^C}}\psi\right\|_{\mathcal{C}}.
\end{aligned}
\end{equation}
By Theorem 4.1 in our recent work \cite{hc},  we have the following squeezing property of the first part of  \eqref{5.14a}.
\begin{lem}\label{lem5.1}Let $P$ be  the finite dimension projection $P_{k_m}$ defined by \eqref{3.13}, $\varrho_{1}, \varrho_{m}$  and $K_m$ being defined in \eqref{3.11} and \eqref{3.13} respectively. Then we have
 \begin{equation}\label{3.16}
\left\|P \Phi(t)|_{\mathcal{C}_{\Omega_K}}\varphi-P \Phi(t)|_{\mathcal{C}_{\Omega_K}}\psi\right\|_{\mathcal{C}_{\Omega_K}} \leq 2e^{(L_f+\varrho_1)t}\left\|\varphi-\psi\right\|_{\mathcal{C}}
\end{equation}
and
 \begin{equation}\label{3.17}
\begin{gathered}
\left\|(I-P) \Phi(t)|_{\mathcal{C}_{\Omega_K}}\varphi-(I-P) \Phi(t)|_{\mathcal{C}_{\Omega_K}}\psi\right\|_{\mathcal{C}_{\Omega_K}}\leq (K_me^{\varrho_m t}+\frac{K_mL_f }{\varrho_1+L_f-\varrho_m} e^{(L_f+\varrho_1)t})\left\|\varphi-\psi\right\|_{\mathcal{C}}
\end{gathered}
\end{equation}
for any $t\geq 0$ and $\varphi, \psi\in\mathcal{A}$.
\end{lem}
In the sequel, we show the squeezing property of the second part of  \eqref{5.14a}.
\begin{lem}\label{lem5.2}
Let $T_{\mathcal{D}}$ as well as $R$ be defined in Lemma \ref{lem2.2}. Then,   it holds that \begin{equation}\label{5.16}
\begin{aligned}
 \left\| \Phi(t)|_{\mathcal{C}_{\Omega_K^C}}\varphi- \Phi(t)|_{\mathcal{C}_{\Omega_K^C}}\psi\right\|_{\mathcal{C}}
& =\|w^\phi_t-w^\psi_t\|_{\mathcal{C}}
  \leq  \sqrt{c_2}e^{\frac{1}{2}[c_2(\sigma+ L_{f}^{2})-(\mu-\sigma-1)]t} \|\phi-\varphi\|_{\mathcal{C}} .
\end{aligned}
\end{equation}
\end{lem}
\begin{proof} Take $\bar{R}=\max\{2, R\}$ and denote by $y=w^\phi_t-w^\psi_t$. Then it follows from  \eqref{2.3} that $y$ is the solution of the following equation
\begin{equation}\label{5.14}
\left\{\begin{array}{l}
\dfrac{\partial y(x,t)}{\partial t}=\Delta y(x,t)-\mu y(x,t)+\sigma y(x,t-\tau)+[f \left(u^\phi(x,t-\tau)\right)-f \left(u^\psi(x,t-\tau)\right)]\chi_{\Omega_{\bar{R}}^C}(x), \\
y(x,s)=[\phi(x, s)-\psi(x,s)] \chi_{\Omega_{\bar{R}}^C}(x), s \in[-\tau, 0], x \in \mathbb{R}^N.
\end{array}\right.
\end{equation}
 Multiplying  both sides of \eqref{5.14} by $\chi\left(\frac{|x|^2}{\bar{R}}\right)y$ and integrating on $\Omega_{\bar{R}}^C$ imply
\begin{equation}\label{5.15}
\begin{aligned}
&\frac{1}{2} \frac{\mathrm{d}}{\mathrm{d} t} \int_{\Omega_{\bar{R}}^C}\chi\left(\frac{|x|^2}{\bar{R}}\right)y^2(t)dx\\
&\quad=\int_{\Omega_{\bar{R}}^C} \chi\left(\frac{|x|^2}{\bar{R}}\right)y \Delta y dx-\mu \int_{\Omega_{\bar{R}}^C} \chi\left(\frac{|x|^2}{\bar{R}}\right)y^2(t)dx \\
&\qquad+\int_{\Omega_{\bar{R}}^C} \chi\left(\frac{|x|^2}{\bar{R}}\right)\sigma y(t-\tau) y(t) \mathrm{d}x  \\
&\qquad+ \int_{\Omega_{\bar{R}}^C} \chi\left(\frac{|x|^2}{\bar{R}}\right)[f \left(u^\phi(x,t-\tau)\right)-f \left(u^\psi(x,t-\tau)\right)]\chi_{\Omega_{\bar{R}}^C}(x) y(t) \mathrm{d}x.
\end{aligned}
\end{equation}
We estimate each term on the right hand side of \eqref{5.15} as follows. First, we have
\begin{equation}\label{5.25}
\begin{aligned}
\int_{\Omega_{\bar{R}}^C} \chi\left(\frac{|x|^2}{\bar{R}}\right) u(t) \Delta u(t) d x= &-\int_{\Omega_{\bar{R}}^C} \chi\left(\frac{|x|^2}{\bar{R}}\right)|\nabla u(t)|^2 d x.
\end{aligned}
\end{equation}
Next, by the Young inequality, we deduce
\begin{equation}\label{5.28}
\begin{aligned}
& \sigma\int_{\Omega_{\bar{R}}^C}  \chi\left(\frac{|x|^2}{\bar{R}}\right) y(t-\tau) y(t) \mathrm{d} x \\
&\quad \leq  \sigma \int_{\Omega_{\bar{R}}^C}  \chi\left(\frac{|x|^2}{\bar{R}}\right)[\frac{1}{2} |y(t-\tau)|^2+\frac{1}{2}|y(t)|^2 ]\mathrm{d} x\\
 &\quad \leq  \frac{\sigma}{2}  \int_{\Omega_{\bar{R}}^C}  \chi\left(\frac{|x|^2}{\bar{R}}\right)|y(t-\tau)|^2\mathrm{d}x+\frac{\sigma}{2}\int_{\Omega_{\bar{R}}^C}  \chi\left(\frac{|x|^2}{\bar{R}}\right)|y(t)|^2\mathrm{d}x,
\end{aligned}
\end{equation}
and
\begin{equation}\label{5.29}
\begin{aligned}
 &\int_{\Omega_{\bar{R}}^C}\chi\left(\frac{|x|^2}{\bar{R}}\right) [f \left(u^\phi(x,t-\tau)\right)-f \left(u^\psi(x,t-\tau)\right)]\chi_{\Omega_K^C}(x)y(t)\mathrm{d} x\\
  & \leq   \frac{1}{2}\int_{\Omega_{\bar{R}}^C}  \chi\left(\frac{|x|^2}{\bar{R}}\right) |[f \left(u^\phi(x,t-\tau)\right)-f \left(u^\psi(x,t-\tau)\right)]\chi_{\Omega_K^C}(x)|^{2}\mathrm{d} x \\
  &\quad+\frac{1}{2}\int_{\Omega_{\bar{R}}^C}  \chi\left(\frac{|x|^2}{\bar{R}}\right) |y(t)|^{2}\mathrm{d} x  \\
& \leq  \frac{1}{2}L_{f}^{2}\int_{\Omega_{\bar{R}}^C}  \chi\left(\frac{|x|^2}{\bar{R}}\right)  |y(t-\tau)|^{2} \mathrm{d} x +\frac{1}{2}\int_{\Omega_{\bar{R}}^C}  \chi\left(\frac{|x|^2}{\bar{R}}\right) |y(t)|^{2}\mathrm{d} x.
\end{aligned}
\end{equation}
Incorporating  \eqref{5.28} and \eqref{5.29} into \eqref{5.15} yields
\begin{equation}\label{5.31}
\begin{aligned}
\frac{\mathrm{d}}{\mathrm{d} t} \int_{\Omega_{\bar{R}}^C}\chi\left(\frac{|x|^2}{\bar{R}}\right)|y(t)|^2dx&\leq (1+\sigma-\mu)\int_{\Omega_{\bar{R}}^C} \chi\left(\frac{|x|^2}{\bar{R}}\right)|y(t)|^2 dx \\
&\quad+( \sigma + L_f^2)  \int_{\Omega_{\bar{R}}^C}  \chi\left(\frac{|x|^2}{\bar{R}}\right)|y(t-\tau)|^2\mathrm{d}x.
\end{aligned}
\end{equation}
Thanks to Gr{o}nwall's inequality on $[0, t]$, we have
 \begin{equation}\label{5.21}
\begin{aligned}
\int_{\Omega_{\bar{R}}^C}\chi\left(\frac{|x|^2}{\bar{R}}\right)|y(t)|^2\mathrm{d} x&\leq e^{(\sigma-\mu+1) t}\int_{\Omega_{\bar{R}}^C}\chi\left(\frac{|x|^2}{\bar{R}}\right)|y(0)|^2\mathrm{d} x+( \sigma +L_f^2) \int_{0}^te^{(\sigma-\mu+1) (t-s)}\\
& \int_{\Omega_{\bar{R}}^C}  \chi\left(\frac{|x|^2}{\bar{R}}\right)|y(s-\tau)|^2\mathrm{d}x\mathrm{d}s.
\end{aligned}
\end{equation}
 This implies  that, for all  $t\in[0, \infty)$ and $\zeta \in[-\tau, 0]$,
 \begin{equation}\label{5.22}
\begin{aligned}
\int_{\Omega_{\bar{R}}^C}\chi\left(\frac{|x|^2}{\bar{R}}\right)|y(t+\zeta)|^2\mathrm{d} x&\leq e^{(\sigma-\mu+1) (t+\zeta)}\int_{\Omega_{\bar{R}}^C}\chi\left(\frac{|x|^2}{\bar{R}}\right)|y(0)|^2\mathrm{d} x\\&+( \sigma +L_f^2) \int_{0}^{t+\zeta}e^{(\sigma-\mu+1) (t+\zeta-s)}
 \int_{\Omega_{\bar{R}}^C}  \chi\left(\frac{|x|^2}{\bar{R}}\right)|y(s-\tau)|^2\mathrm{d}x\mathrm{d}s \\&\leq c_2e^{(\sigma-\mu+1)t}\int_{\Omega_{\bar{R}}^C}\chi\left(\frac{|x|^2}{\bar{R}}\right)|y(0)|^2\mathrm{d} x\\
 &+c_2(\sigma +L_f^2) \int_{0}^{t}e^{(\sigma-\mu+1) (t-s)} \int_{\Omega_{\bar{R}}^C}  \chi\left(\frac{|x|^2}{\bar{R}}\right)|y(s-\tau)|^2\mathrm{d}x\mathrm{d}s.
\end{aligned}
\end{equation}
Notice that $\left\|y_{t}\right\|_{\mathcal{C}}=\sup \left\{\|y(t+\zeta)\|: \zeta \in[-\tau, 0]\right\},$ whence
\begin{equation}\label{5.23}
\begin{aligned}
\|y_t\|_{\mathcal{C}}^{2}& \leq c_2e^{(\sigma-\mu+1) t}\|\phi-\varphi\|_{\mathcal{C}}^{2}+  c_2(\sigma+ L_{f}^{2})\int_{0}^{t} e^{(\sigma-\mu+1) (t-s)}\| y_s \|_{_{\mathcal{C}}}^{2}\mathrm{d}s.
\end{aligned}
\end{equation}
Multiplying both sides of \eqref{5.23} by  $e^{(\mu-\sigma-1)t}$, we derive
\begin{equation}\label{5.24}
\begin{aligned}
e^{(\mu-\sigma-1)t}\|y_t\|_{\mathcal{C}}^{2}& \leq c_2 \|\phi-\varphi\|_{\mathcal{C}}^{2}+  c_2(\sigma+ L_{f}^{2})\int_{0}^{t} e^{-(\sigma-\mu+1) s}\| y_s \|_{_{\mathcal{C}}}^{2}\mathrm{d}s.
\end{aligned}
\end{equation}
Using Gr{o}nwall's inequality again yields
\begin{equation}\label{5.24a}
\begin{aligned}
e^{(\mu-\sigma-1)t}\|y_t\|_{\mathcal{C}}^{2}& \leq c_2e^{c_2(\sigma+ L_{f}^{2})t} \|\phi-\varphi\|_{\mathcal{C}}^{2},
\end{aligned}
\end{equation}
indicating that
\begin{equation}\label{5.24b}
\begin{aligned}
 \|y_t\|_{\mathcal{C} }^{2} \leq c_2e^{[c_2(\sigma+ L_{f}^{2})-(\mu-\sigma-1)]t} \|\phi-\varphi\|_{\mathcal{C} }^{2}\leq c_2e^{[c_2(\sigma+ L_{f}^{2})-(\mu-\sigma-1)]t} \|\phi-\varphi\|_{\mathcal{C}}^{2}.
\end{aligned}
\end{equation}
The proof is complete.
\end{proof}
\section{Hausdorff  and  fractal dimensions}
This section is devoted to the estimations of Hausdorff and  fractal dimensions of the global attractor obtained in Theorem \ref{thm4.1}. We first concentrate on the Hausdorff dimension and  review the definition  of Hausdorff dimension of global attractors for autonomous dynamical systems.

Let  $\mathcal{A}$  be the global attractor of \eqref{1} obtained in Theorem \ref{thm4.1}.  The Hausdorff dimension of the compact set $\mathcal{A}\subset \mathcal{C}$ is
$$
d_{\mathcal{C}}(\mathcal{A})=\inf \left\{d: \mu_{\mathcal{C}}(\mathcal{A}, d)= 0 \right\}
$$
where, for $d \geq 0$,
$$\mu_{\mathcal{C}}(\mathcal{A}, d)=\lim _{\varepsilon \rightarrow 0} \mu_{\mathcal{C}}(\mathcal{A}, d, \varepsilon)$$
 denotes the $d$-dimensional Hausdorff measure of the set $\mathcal{A}\subset \mathcal{C}$, where
 $$\mu_{\mathcal{C}}(\mathcal{A}, d, \varepsilon)=\inf \sum_{i} r_{i}^{d}$$
 and the infimum is taken over all coverings of $\mathcal{A}$ by balls of radius $r_{i} \leqslant \varepsilon$. It can be shown that there exists $d_{\mathcal{C}}(\mathcal{A}) \in[0,+\infty]$ such that $\mu_{\mathcal{C}}(\mathcal{A}, d)=0$ for $d>d_{\mathcal{C}}(\mathcal{A})$ and $\mu_{\mathcal{C}}(\mathcal{A}, d)=\infty$ for $d<d_{\mathcal{C}}(\mathcal{A})$. $d_{\mathcal{C}}(\mathcal{A})$ is called the Hausdorff dimension of $\mathcal{A}$.

For any $\phi\in \mathcal{C}$, define the map $P_{\chi(\Omega_K^C)}: \mathcal{C}\rightarrow \mathcal{C} $ and $P_{\chi(\Omega_K)}: \mathcal{C}\rightarrow \mathcal{C}_{\Omega_K}$ by
\begin{equation}\label{6.1a}
\begin{aligned}
P_{\chi(\Omega_K)} \phi=\chi(\Omega_K)\phi
\end{aligned}
\end{equation}
and
\begin{equation}\label{6.1}
\begin{aligned}
P_{\chi(\Omega_K^C)} \phi=\chi(\Omega_K^C)\phi
\end{aligned}
\end{equation}
respectively, where $\chi(\Omega_K)$ and $\chi(\Omega_K^C)$ are the characteristic functions defined in \eqref{2.1} and \eqref{2.2}.
Then, by \eqref{3.12a}, we have
\begin{equation}\label{6.2}
\begin{aligned}
\mathcal{C}=P_{k_m}P_{\chi(\Omega_K)}\mathcal{C}  \bigoplus  Q_{k_m}P_{\chi(\Omega_K)}\mathcal{C} \bigoplus P_{\chi(\Omega_K^C)}\mathcal{C} \triangleq  \mathcal{C}_{\Omega_K}^U \bigoplus  \mathcal{C}_{\Omega_K}^S \bigoplus \mathcal{C}_{\Omega_K^C}.
\end{aligned}
\end{equation}
where $P_{k_m}$ and $Q_{k_m}$ are defined by \eqref{3.12b} and \eqref{3.12c}.
Define
\begin{equation}\label{6.2a}
\begin{aligned}
\mathcal{P}=P_{k_m}P_{\chi(\Omega_K)}, \mathcal{Q}=Q_{k_m}P_{\chi(\Omega_K)},\mathcal{R}= P_{\chi(\Omega_K^C)}.
\end{aligned}
\end{equation}
It follows from  lemmas  \ref{lem5.1} and \ref{lem5.2} that
\begin{equation}\label{6.3}
\begin{aligned}
\|\mathcal{P}[\Phi(t,\varphi)-\Phi(t,\psi)]\|_{\mathcal{C}(\Omega_K)} \leq e^{(L_f+\varrho_1)t}\|\varphi-\psi\|_{\mathcal{C}},
\end{aligned}
\end{equation}
 \begin{equation}\label{6.4}
\begin{aligned}
\|\mathcal{Q}[\Phi(t,\varphi)-\Phi(t,\psi)]\|_{\mathcal{C}(\Omega_K)} \leq (K_me^{\varrho_m t}+\frac{K_mL_f }{\varrho_1+L_f-\varrho_m} e^{(L_f+\varrho_1)t})\|\varphi-\psi\|_{\mathcal{C}},
\end{aligned}
\end{equation}
and
\begin{equation}\label{6.5}
\begin{aligned}
\|\mathcal{R}[\Phi(t,\varphi)-\Phi(t,\psi)]\|_{\mathcal{C}(\Omega_K^C)} \leq \sqrt{c_2}e^{\frac{1}{2}[c_2(\sigma+ L_{f}^{2})-(\mu-\sigma-1)]t}\|\varphi-\psi\|_{\mathcal{C}}.
\end{aligned}
\end{equation}
Moreover, by \eqref{3.12} and \eqref{3.12a}, we can see $\mathcal{P}$ has $k_m$ dimension range space, that is $P_{k_m}\mathcal{C}_{\Omega_K}$ is a  $k_m$ dimension subspace of $\mathcal{C}$.
In the following, we prove   that the global attractor  of \eqref{1} established in Theorem \ref{thm4.1} has finite Hausdorff and fractal dimension.
\begin{thm}\label{thm5.2} Let $k_m, \varrho_{1}, \varrho_{m}$ and $K_m$ be  defined in \eqref{3.11}  and \eqref{3.13} respectively,  $\mathcal{P}, \mathcal{Q}$ and $\mathcal{R}$ be  defined by \eqref{6.2a}. Assume that conditions of Lemma \ref{lem5.1} are satisfied  and there exist $t_0>0$ and $0<\alpha<2$ such that
 \begin{equation}\label{6.6}
\eta= 2K_me^{\varrho_mt_0}+(\alpha+\frac{2K_mL_f }{\varrho_1+L_f-\varrho_m}) e^{(L_f+\varrho_1)t_0}+2\sqrt{c_2}e^{\frac{[c_2(\sigma+ L_{f}^{2})-(\mu-\sigma-1)]t_0}{2}}<1.
\end{equation}
Then, the Hausdorff dimension of global attractor $\mathcal{A}$ satisfies
 \begin{equation}\label{6.7}
d<\frac{-\ln k_m-k_m\ln (2+\frac{4}{\alpha})}{\ln \left(2K_me^{\varrho_mt_0}+(\alpha+\frac{2K_mL_f }{\varrho_1+L_f-\varrho_m}) e^{(L_f+\varrho_1)t_0}+2\sqrt{c_2}e^{\frac{[c_2(\sigma+ L_{f}^{2})-(\mu-\sigma-1)]t_0}{2}}\right)}.
\end{equation}
\end{thm}
\begin{proof}
Since $\mathcal{A}$ is a compact subset of $\mathcal{C}$, for any $0<\varepsilon<1$, there exist $r_1, \ldots, r_N$ in $(0, \varepsilon]$ and $\tilde{u}_1, \ldots, \tilde{u}_N$ in $\mathcal{C}$ such that
 \begin{equation}\label{6.8}
\begin{gathered}
\mathcal{A}\subset \bigcup_{i=1}^N B\left(\tilde{u}_i, r_i\right),
\end{gathered}
\end{equation}
where $B(\tilde{u}_i, r_i)$ represents the ball in $\mathcal{C}$ of center $\tilde{u}_i$ and radius $r_i$. Without loss of generality, we can assume that for any $i$
 \begin{equation}\label{6.9}
\begin{gathered}
B\left(\tilde{u}_i, r_i\right) \cap \mathcal{A} \neq \emptyset,
\end{gathered}
\end{equation}
otherwise, it can be deleted from the sequence $\tilde{u}_1, \ldots, \tilde{u}_N$. Therefore, we can choose  $u_i, i=1,2, \cdots, N$ such that
 \begin{equation}\label{6.10}
\begin{gathered}
u_i \in B\left(\tilde{u}_i, r_i\right) \cap \mathcal{A},
\end{gathered}
\end{equation}
and
 \begin{equation}\label{6.11}
\begin{gathered}
\mathcal{A} \subset \bigcup_{i=1}^N\left(B\left(u_i, 2 r_i\right) \cap \mathcal{A}\right).
\end{gathered}
\end{equation}
It follows from \eqref{6.3} and \eqref{6.4} that, for any $t_0>0$ and $u\in B\left(u_i, 2 r_i\right) \cap \mathcal{A}$,
 \begin{equation}\label{6.12}
\begin{gathered}
\left\|\mathcal{P}\Phi(t_0) u- \mathcal{P} \Phi\left(t_0\right) u_i\right\|_{\mathcal{C}} \leq 2e^{(L_f+\varrho_1)t_0} r_i,
\end{gathered}
\end{equation}
 \begin{equation}\label{6.13}
\begin{gathered}
\left\|\mathcal{Q}\Phi(t_0) u- \mathcal{Q} \Phi\left(t_0\right) u_i\right\|_{\mathcal{C}} \leq 2(K_me^{\varrho_m t_0}+\frac{K_mL_f }{\varrho_1+L_f-\varrho_m} e^{(L_f+\varrho_1)t_0})r_i,
\end{gathered}
\end{equation}
and
 \begin{equation}\label{6.13a}
\begin{gathered}
\left\|\mathcal{R}\Phi(t_0) u- \mathcal{R} \Phi\left(t_0\right) u_i\right\|_{\mathcal{C}} \leq 2\sqrt{c_2}e^{\frac{[c_2(\sigma+ L_{f}^{2})-(\mu-\sigma-1)]t_0}{2}}r_i.
\end{gathered}
\end{equation}
By Lemma \ref{lem2.2}, for any $\alpha>0$, we can find $y_i^1, \ldots, y_i^{n_i}$ such that
 \begin{equation}\label{6.14}
\begin{gathered}
B_{\mathcal{P} \mathcal{C}}\left(\mathcal{P} \Phi\left(t_0\right) u_i, 2e^{(L_f+\varrho_1)t_0} r_i\right) \subset \bigcup_{j=1}^{n_i} B_{\mathcal{P}\mathcal{C}}\left(y_i^j, \alpha e^{(L_f+\varrho_1)t_0} r_i\right)
\end{gathered}
\end{equation}
with
 \begin{equation}\label{6.15}
\begin{gathered}
n_i \leq k_m \left(2+ \frac{4}{\alpha}  \right)^{k_m},
\end{gathered}
\end{equation}
where $k_m$ is the dimension of $\mathcal{P }\mathcal{C}$ and we have denoted by $B_{\mathcal{P}\mathcal{C}}(y, r)$ the ball in $\mathcal{P}\mathcal{C}$ of radius $r$ and center $y$.

Set
 \begin{equation}\label{6.16}
\begin{gathered}
u_i^j=y_i^j+\mathcal{Q} \Phi\left(t_0\right) u_i+\mathcal{R}\Phi\left(t_0\right) u_i
\end{gathered}
\end{equation}
for $i=1, \ldots, N, j=1, \ldots, n_i$. Then, for any $u\in B\left(u_i, 2 r_i\right) \cap \mathcal{A}$, there exists a $j$ such that
 \begin{equation}\label{6.17}
\begin{aligned}
&\left\|\Phi\left(t_0\right) u-u_i^j\right\|_{\mathcal{C}}\\
&\quad \leq\left\|\mathcal{P} \Phi\left(t_0\right) u-y_i^j\right\|_{\mathcal{C}}+\left\|\mathcal{Q} \Phi\left(t_0\right) u-\mathcal{Q} \Phi\left(t_0\right) u_i\right\|_{\mathcal{C}}+\left\|\mathcal{R} \Phi\left(t_0\right) u-\mathcal{R} \Phi\left(t_0\right) u_i\right\|_{\mathcal{C}} \\
&\quad \leq\left(2K_me^{\varrho_mt_0}+(\alpha+\frac{2K_mL_f }{\varrho_1+L_f-\varrho_m}) e^{(L_f+\varrho_1)t_0}+2\sqrt{c_2}e^{\frac{[c_2(\sigma+ L_{f}^{2})-(\mu-\sigma-1)]t_0}{2}}\right) r_i.
\end{aligned}
\end{equation}
Denote by $\eta=\left(2K_me^{\varrho_mt_0}+(\alpha+\frac{2K_mL_f }{\varrho_1+L_f-\varrho_m}) e^{(L_f+\varrho_1)t_0}+2\sqrt{c_2}e^{\frac{[c_2(\sigma+ L_{f}^{2})-(\mu-\sigma-1)]t_0}{2}}\right)$. We then have
 \begin{equation}\label{6.22}
\Phi\left(1\right)\left(B\left(u_i, 2 r_i\right) \cap \mathcal{A}\right) \subset \bigcup_{j=1}^{n_i} B\left(u_i^j, \eta r_i\right).
\end{equation}
Thanks to the invariance of $\mathcal{A}$, i.e., $\mathcal{A}=\Phi\left(1\right) \mathcal{A}$, we have
 \begin{equation}\label{6.23}
\mathcal{A} \subset \bigcup_{i=1}^N \bigcup_{j=1}^{n_i} B\left(u_i^j, \eta r_i\right) .
\end{equation}
This implies that, for any $d \geq 0$,
 \begin{equation}\label{6.24}
\begin{aligned}
\mu_H\left(\mathcal{A}, d, \eta\varepsilon\right)
\leq \sum_{i=1}^N \sum_{j=1}^{n_i} \eta^{d}r_i^d \leq k_m (2+\frac{4}{\alpha})^{k_m} \eta^{d} \sum_{i=1}^N r_i^d.
\end{aligned}
\end{equation}
We deduce, by taking the infimum over all the coverings of $\mathcal{A}$ by balls of radii less than $\varepsilon$,
 \begin{equation}\label{6.25}
\begin{aligned}
\mu_H\left(\mathcal{A}, d, \eta\varepsilon\right)\leq k_m (2+\frac{4}{\alpha})^{k_m} \eta^{d}\mu_H(\mathcal{A}, d, \varepsilon).
\end{aligned}
\end{equation}
Applying the formula recursively for $k$ times we obtain
 \begin{equation}\label{6.28}
\begin{aligned}
\mu_H\left(\mathcal{A}, d, (\eta\varepsilon)^k\right)\leq [k_m (2+\frac{4}{\alpha})^{k_m} \eta^{d}]^k\mu_H(\mathcal{A}, d, \varepsilon).
\end{aligned}
\end{equation}
Therefore, if
 \begin{equation}\label{6.26}
d<\frac{-\ln k_m-k_m\ln (2+\frac{4}{\alpha})}{\ln \left(2K_me^{\varrho_mt_0}+(\alpha+\frac{2K_mL_f }{\varrho_1+L_f-\varrho_m}) e^{(L_f+\varrho_1)t_0}+2\sqrt{c_2}e^{\frac{[c_2(\sigma+ L_{f}^{2})-(\mu-\sigma-1)]t_0}{2}}\right)},
\end{equation}
then
 \begin{equation}\label{6.27}
k_m (2+\frac{4}{\alpha})^{k_m}  \eta^{d}<1.
\end{equation}
Thus, by taking $k \rightarrow \infty$, we have $(\eta\varepsilon)^k\rightarrow 0$
and \eqref{6.28} leads to
 \begin{equation}\label{6.31}
\mu_H(\mathcal{A}, d, (\eta\varepsilon)^k) \rightarrow 0.
\end{equation}
This completes the proof.
\end{proof}
\begin{rem}\label{rem5.2}
Since $\varrho_1$ and $\varrho_m$ represent the first and the $m$-th eigenvalues of the linear part $A_U$ of Eq. \eqref{5.1}, which depends on the delay effect, we can see the Hausdorff dimension of global attractor $\mathcal{A}$ of Eq. \eqref{5.1} depends on  the time delay via the distribution of eigenvalues of the linear part $A_U$ of Eq. \eqref{5.1}. Furthermore, it follows from \eqref{6.26} that the Hausdorff dimension depends on  a variable parameter $\alpha$, constants of exponential dichotomy, the Lipschitz constant of the nonlinear term and the spectrum gap of the linear part $A_U$, indicating that the Hausdorff dimension of global attractor $\mathcal{A}$ is very flexible to be tuned by a variety of parameters.
\end{rem}

Next, we study the fractal dimension of attractors for the nonlinear dynamical system  $\Phi(t)$. The fractal dimension (or capacity) of $\mathcal{A}$ is defined as
 \begin{equation}\label{6.33}
\operatorname{dim}_f \mathcal{A}=\limsup_{\varepsilon \rightarrow 0} \frac{\ln N_{\varepsilon}(\mathcal{A})}{-\ln \varepsilon},
\end{equation}
where $N_{\varepsilon}(\mathcal{A})$ is the minimum number of balls of radius less than $\varepsilon$  needed to cover $\mathcal{A}$.
\begin{thm}\label{thm5.3}
Let $k_m, \varrho_{1}, \varrho_{m}$ and $K_m$ be  defined in \eqref{3.11}  and \eqref{3.13} respectively,  $\mathcal{P}, \mathcal{Q}$ and $\mathcal{R}$ be  defined by \eqref{6.2a}. Assume that conditions of Lemma \ref{lem5.1} are satisfied.  Moreover, assume   there exists $\beta>0$ such that $\zeta:=\beta e^{(L_f+\varrho_1)}+K_me^{\varrho_m}+\frac{K_mL_f }{\varrho_1+L_f-\varrho_m}e^{(L_f+\varrho_1)}+\sqrt{c_2}e^{\frac{[c_2(\sigma+ L_{f}^{2})-(\mu-\sigma-1)]}{2}}<1$. Then, the fractal dimension of global attractor $\mathcal{A}$  has an upper bound
 \begin{equation}\label{6.34}
\operatorname{dim}_f \mathcal{A}\leq \frac{\ln k_m +k_m \ln(2+\frac{2}{\beta})}{-\ln \zeta}<\infty.
\end{equation}
\end{thm}
\begin{proof}
Since $\mathcal{A}$ is the global attractor, then it is compact and hence the number $R_\mathcal{A}$ is well defined. Thus, for any $u_0 \in \mathcal{A}$, we have
 \begin{equation}\label{6.35}
\mathcal{A} \subseteq B\left(u_0, R_\mathcal{A}\right),
\end{equation}
where $B\left(u_0,R_\mathcal{A}\right)$ is the ball with center $u_0$ and radius $R_\mathcal{A}$. For any $u \in \mathcal{A} \cap B\left(u_0, R_\mathcal{A}\right)$, it follows from lemmas \ref{lem5.1} and \ref{lem5.2}, \eqref{6.3} and \eqref{6.4} that for any $t_0>0$ and $u\in B\left(u_i, 2 r_i\right) \cap \mathcal{A}$,
 \begin{equation}\label{6.36}
\begin{gathered}
\left\|\mathcal{P}\Phi(t_0) u- \mathcal{P} \Phi\left(t_0\right) u_i\right\|_{\mathcal{C}} \leq  e^{(L_f+\varrho_1)t_0} r_i,
\end{gathered}
\end{equation}
 \begin{equation}\label{6.37}
\begin{gathered}
\left\|\mathcal{Q}\Phi(t_0) u- \mathcal{Q} \Phi\left(t_0\right) u_i\right\|_{\mathcal{C}} \leq  (K_me^{\varrho_m t_0}+\frac{K_mL_f }{\varrho_1+L_f-\varrho_m} e^{(L_f+\varrho_1)t_0})r_i,
\end{gathered}
\end{equation}
and
 \begin{equation}\label{6.37a}
\begin{gathered}
\left\|\mathcal{R}\Phi(t_0) u- \mathcal{R} \Phi\left(t_0\right) u_i\right\|_{\mathcal{C}} \leq \sqrt{c_2}e^{\frac{[c_2(\sigma+ L_{f}^{2})-(\mu-\sigma-1)]t_0}{2}}r_i.
\end{gathered}
\end{equation}
By Lemma \ref{lem2.2}, for any $\beta>0$, we can find $y_i^1, \ldots, y_i^{n_i}$ such that
 \begin{equation}\label{6.38}
\begin{gathered}
B_{\mathcal{P} \mathcal{C}}\left(\mathcal{P} \Phi\left(t_0\right) u_i,  e^{(L_f+\varrho_1)} r_i\right) \subset \bigcup_{j=1}^{n_i} B_{\mathcal{P}\mathcal{C}}\left(y_i^j,  \beta e^{(L_f+\varrho_1)}r_i\right)
\end{gathered}
\end{equation}
with
 \begin{equation}\label{6.15}
\begin{gathered}
n_i \leq k_m 2^{k_m} \left(1+ \frac{1}{\beta} \right)^{k_m},
\end{gathered}
\end{equation}
where $k_m$ is the dimension of $\mathcal{P }\mathcal{C}$ and we have denoted by $B_{\mathcal{P}\mathcal{C}}(y, r)$ the ball in $\mathcal{P}\mathcal{C}$ of radius $r$ and center $y$.

Take $t_0=1$ and set
 \begin{equation}\label{6.40}
\begin{gathered}
u_0^j=y_0^j+\mathcal{Q }\Phi\left(1\right) u_0+\mathcal{R} \Phi\left(1\right) u_0
\end{gathered}
\end{equation}
for $ j=1, \ldots, n_0$. Then, for any $u \in \mathcal{A} \cap B\left(u_0, R_\mathcal{A}\right)$, there exists $j$ such that
 \begin{equation}\label{6.41}
\begin{aligned}
&\left\|\Phi\left(1\right) u-u_0^j\right\|\\
&\quad \leq\left\|\mathcal{P}\Phi\left(1\right)  u-y_0^j\right\|+\left\|\mathcal{Q } \Phi\left(1\right)  u-\mathcal{Q } \Phi\left(1\right) u_0\right\|+ \left\|\mathcal{R} \Phi\left(1\right)  u-\mathcal{R} \Phi\left(1\right) u_0\right\|\\
&\quad \leq\left(\beta e^{(L_f+\varrho_1)}+K_me^{\varrho_m}+\frac{K_mL_f }{\varrho_1+L_f-\varrho_m}e^{(L_f+\varrho_1)}+\sqrt{c_2}e^{\frac{[c_2(\sigma+ L_{f}^{2})-(\mu-\sigma-1)]}{2}}\right)R_\mathcal{A}.
\end{aligned}
\end{equation}
Denote by
$\zeta=\beta e^{(L_f+\varrho_1)}+K_me^{\varrho_m}+\frac{K_mL_f }{\varrho_1+L_f-\varrho_m}e^{(L_f+\varrho_1)}+\sqrt{c_2}e^{\frac{[c_2(\sigma+ L_{f}^{2})-(\mu-\sigma-1)]}{2}}$. Since $\mathcal{A}$ is invariant,  i.e., $\mathcal{A}=S\left(1\right) \mathcal{A}$, we have
 \begin{equation}\label{6.43}
\begin{aligned}
\mathcal{A} & =\Phi\left(1\right)\left(\mathcal{A}\cap B\left(u_0, R_\mathcal{A}\right)\right) & \subseteq \bigcup_{j=1}^{n_0} B\left(u_{0}^j,\zeta R_\mathcal{A}\right).
\end{aligned}
\end{equation}
Applying the formula recursively for $k$ times gives
 \begin{equation}\label{6.44}
\begin{aligned}
\mathcal{A} & =\Phi\left(k \right)\left(\mathcal{A}\cap B\left(u_0, R_\mathcal{A}\right)\right) & \subseteq \bigcup_{j=1}^{n_0,n_1,\cdots, n_{k-1}} B\left(u_{k-1}^j,\zeta^k R_\mathcal{A}\right),
\end{aligned}
\end{equation}
implying that the minimal number $N_{r_k}\left(\mathcal{A}\right)$ of balls with radius $r_k=\zeta^k R_\mathcal{A}$ covering $\mathcal{A}$ in $\mathcal{C}$ satisfies
 \begin{equation}\label{6.44a}
\begin{aligned}
N_{r_k}\left(\mathcal{A}\right) \leq n_0 \cdot \ldots \cdot n_{k-1} \leq[k_m 2^{k_m} \left(1+\frac{1}{\beta}\right)^{k_m}]^k.
\end{aligned}
\end{equation}
Since we have assumed that $\zeta <1$, then $r_k\rightarrow 0$ as $k\rightarrow \infty$.
Then it follows from \eqref{6.33} that
 \begin{equation}\label{6.45}
\begin{aligned}
\operatorname{dim}_f \mathcal{A} & =\limsup_{r_k\rightarrow 0} \frac{\ln N_{r_k}(\mathcal{A})}{-\ln r_k}\\
& \leq \limsup_{k\rightarrow \infty} \frac{\ln [k_m 2^{k_m} \left(1+\frac{1}{\beta}\right)^{k_m}]^k}{-\ln (\zeta^k R_\mathcal{A})}\\
&=\frac{\ln k_m +k_m \ln(2+\frac{2}{\beta})}{-\ln \zeta}<\infty.
\end{aligned}
\end{equation}
\end{proof}

\section{Summary}
We end this paper by making a detailed comparison of the present work with \cite{CE00,EM99,EM01,C18}. In \cite{EM99}, the authors constructed exponential attractors of a reaction-diffusion equation in a weighted Hilbert space without giving explicit fractal dimension of the attractor by the squeeze method originated from \cite{EFNT94}, which cannot be directly to used to to \eqref{1}. In \cite{CE00} and \cite{EM01}, the authors adopted the Lyapunov exponent method to obtain  sharp estimation in a weighted Hilbert space, although the estimations depend on less parameters and may be optimal than  the present work, the methods can be neither be employed to tackle the problem in this paper since $\mathcal{C}$ is not a Hilbert space. In \cite{C18} , the authors obtained exponential attractors with explicit fractal dimensions. Specifically, in \cite{C18}, the obtained fractal dimension depend on a parameter $C$ in formula (8) between two Banach space $H$ and $H_1$, which may be different if different $H$ is differently chosen. As the afore mentioned works concerned about the dimension estimation of uniform attractors, one natural generalization of the present work is to study the nonautonomous case, which will be tackled in an upcoming paper. Moreover, as pointed in \cite{T6}, Lyapunov  exponent method will give shaper estimation of dimensions  compared with the squeeze method and hence another question is how to establish Lyapunov  exponent method  to estimate dimension of attractors in Banach spaces,  which will be studied in the near future.
\section{Acknowledgement}
This work was jointly supported by China Postdoctoral Science Foundation (2019TQ0089), China Scholarship Council(202008430247). \\
The research of T. Caraballo has been partially supported by Spanish Ministerio de Ciencia e
Innovaci\'{o}n (MCI), Agencia Estatal de Investigaci\'{o}n (AEI), Fondo Europeo de
Desarrollo Regional (FEDER) under the project PID2021-122991NB-C21\\
This work began when Wenjie Hu was visiting the Universidad de Sevilla as a visiting scholar, and he would like to thank the staff in the Facultad de Matem\'{a}ticas  for their hospitality and thank the university for its excellent facilities and support during his stay.

\end{document}